\documentclass{amsart}
\usepackage{amsmath,amssymb,amsthm}
\usepackage{algorithm,algpseudocode}
\usepackage{graphicx}
\usepackage[colorlinks=true, allcolors=blue]{hyperref}
\usepackage{mathrsfs}
\usepackage{bm}

\usepackage{color}

\newcommand{\reviewfigure}[2][]{%
  \IfFileExists{#2}{\includegraphics[#1]{#2}}{%
    \fbox{\parbox[c][3.6cm][c]{0.84\textwidth}{\centering
      Original figure not included in this review copy.\\[0.6em]
      \texttt{\detokenize{#2}}}}}}
\usepackage{booktabs}
\usepackage{float}
\usepackage{placeins}
\theoremstyle{plain}
\newtheorem{theorem}{Theorem}[section]
\newtheorem{lemma}[theorem]{Lemma}
\newtheorem{proposition}[theorem]{Proposition}

\theoremstyle{definition}

\newcommand{\RR}{\mathbb{R}}

\newcommand{\dd}{\mathrm{d}}
\newcommand{\compactsum}{\mathop{\textstyle\sum}\limits}

\newcommand{\V}{\mathcal V}
\newcommand{\Vp}{\mathcal V'}
\newcommand{\E}{\mathbb E}
\newcommand{\Pp}{\mathbb P}
\newcommand{\F}{\mathcal F}
\newcommand{\B}{\mathcal B}
\newcommand{\cE}{\mathcal E}
\newcommand{\cM}{\mathcal M}

\DeclareMathOperator{\dist}{dist}
\DeclareMathOperator{\spanop}{span}
\DeclareMathOperator{\tr}{tr}
\DeclareMathOperator{\rank}{rank}

\title[Convergence of Variable-Separation-Based Model Reduction]{
Convergence Analysis of Variable-Separation-Based Model Reduction for Linear Parametric Partial Differential Equations}\thanks{
All authors contributed equally to this work.
This work was supported partly by the National Natural Science Foundation of China (Nos. 12671361, 12471405) and the Science and Technology Innovation Program of Hunan Province (No. 2025RC3080).
}

\author{Liang Chen}
\address{School of Mathematics, Hunan University, Changsha 410082, China.}
\email{chl@hnu.edu.cn}

\author{Qiuqi Li}
\address{School of Mathematics, Hunan University, Changsha 410082, China.}
\email{qli28@hnu.edu.cn}

\author{Yusong Ye}
\address{School of Mathematics, Hunan University, Changsha 410082, China.}
\email{yeyusong@hnu.edu.cn}

\author{Pingwen Zhang}
\address{Institute for Math and AI, Wuhan University, Wuhan 430072, and School of Mathematical Sciences, Peking University, Beijing 100871, China.}
\email{pzhang@pku.edu.cn}


\date{\today}

\dedicatory{}

\begin{document}

\begin{abstract}
This paper analyzes the convergence of a variable-separation (VS) model reduction method that successively constructs separated approximations to the solutions of linear parametric partial differential equations (PDEs) by solving residual problems at sampled parameter values.
We prove that the resulting VS approximations converge almost surely to the exact parametric solution in the parameter-averaged energy norm, under finite-dimensional approximation conditions on the error family and the assumptions of symmetry, uniform coercivity, and uniform boundedness of the governing bilinear forms.
For parametric PDEs whose operators and right-hand sides admit finite-term separated representations, we construct finite-dimensional approximation spaces for the error family and bound their dimensions and approximation errors, thereby verifying these approximation conditions and obtaining an $O((\log N)^{-1})$ upper bound for the expected parameter-averaged error energy after $N$ iterations.
When the solution families are contained in a fixed finite-dimensional space, we obtain the sharper $O(N^{-1})$ bound for the expected error energy and, if the operator is also parameter-independent, exact recovery for almost every parameter after an almost surely finite number of iterations.
Numerical experiments illustrate the error decay of the VS approximations, as well as finite-step recovery in the parameter-independent case.   
\end{abstract}
\maketitle

\section{Introduction}
Parametric partial differential equations (PDEs) arise naturally when material coefficients, geometric configurations, or uncertain inputs are treated as parameters in a mathematical model \cite{CohenDeVore2015,Ghanem2003}.
Such PDE models are indispensable in various engineering disciplines, including structural mechanics, computational fluid dynamics, and electromagnetics \cite{Hesthaven2016}.
In many-query contexts, such as design optimization, inverse problems, and real-time simulation, the high-fidelity model must be evaluated for many parameter instances \cite{Prudhomme2002}.
To mitigate this computational burden, model order reduction (MOR) has been developed to exploit the low-dimensional manifold structure often present in parametric solution sets \cite{Quarteroni2016}.
By replacing the full-order model with a reduced representation, MOR aims to capture the dominant solution behavior while substantially lowering the computational cost \cite{Benner2015}.
To describe the solution family to be approximated, let \(D\subset\RR^d\) be the spatial domain and \(\Omega\subset\RR^m\) the parameter domain, with \(\V\) a suitable Hilbert space for solutions of the underlying PDE.
For each \(\xi\in\Omega\), write \(u(\xi)\in\V\) for the corresponding solution on \(D\), so that the family targeted by MOR is the solution manifold $\cM:=\{u(\xi):\xi\in\Omega\}\subset\V$.
An accurate low-dimensional or separated representation of this manifold allows repeated high-fidelity simulations to be replaced, at least in the online stage, by computations involving only reduced representations.

A widely used class of MOR methods constructs reduced spaces from
high-fidelity solutions and projects the governing equations onto these
spaces.
Reduced basis (RB) methods belong to this class, with early applications
to incompressible viscous flow \cite{Peterson1989}.
The reduced space can be constructed by proper orthogonal decomposition
(POD) \cite{Quarteroni2016}.
This decomposition extracts dominant modes from ensembles of
high-fidelity snapshots, usually through singular value decomposition
\cite{Sirovich1987,BerkoozHolmesLumley1993}.
Alternatively, greedy algorithms select parameter values by maximizing
an error estimator and enrich the reduced space with the corresponding
high-fidelity solutions \cite{Rozza2008}.
Approximations in a reduced spatial space can be written in the
separated form
\[
\begin{array}{ll}
u_N(\xi)
=
\compactsum_{i=1}^{N}\eta_i(\xi)g_i
\qquad
\forall\,\xi\in\Omega,
\end{array}
\]
where \(g_i\in\V\) is a spatial function and \(\eta_i:\Omega\to\RR\) is a parameter-dependent coefficient function.
Proper generalized decomposition (PGD) \cite{Nouy2007,Nouy2010} constructs approximations of this form through successive enrichment, typically determining the new spatial and parameter-dependent factors by alternating or coupled iterative updates.
The variable-separation (VS) method \cite{LiJiang2017} seeks the same separated form through a different enrichment mechanism, in which a residual problem is first solved at a selected parameter to obtain a new spatial direction, and the associated parameter-dependent coefficient function is then determined explicitly.
Introduced for linear stochastic PDEs in \cite{LiJiang2017}, the VS method has since been extended to stochastic saddle-point problems \cite{JiangLi2018} and nonlinear PDEs with random inputs \cite{LiZhang2020}.
Further extensions include stochastic domain decomposition \cite{Chen2024}, parameter-dependent dynamical systems \cite{ChenZhou2025}, and frequency-domain formulations for time-domain Maxwell's equations \cite{BaLi2026}.

For Galerkin approximations in POD spaces, error estimates for linear
and certain nonlinear parabolic problems account for both basis
truncation and time discretization \cite{KunischVolkwein2001}.
The analysis was later extended to nonlinear parabolic systems arising
in fluid dynamics, with estimates that account for the snapshot
locations \cite{KunischVolkwein2002}.
For semilinear parabolic optimal control problems, a posteriori bounds
were also obtained for the distance between locally optimal controls
of the reduced and original problems
\cite{KammannTroltzschVolkwein2013}.
For RB methods, error bounds for the reduced approximation and convergence of the space-enrichment procedure play distinct roles.
A priori convergence results were obtained for single-parameter symmetric coercive elliptic problems \cite{MadayPateraTurinici2002}, while residual-based a posteriori error bounds were developed for parabolic problems with finite-term separable parameter dependence \cite{GreplPatera2005} and for the parametrized steady incompressible Navier--Stokes equations \cite{VeroyPatera2005}.
A comprehensive certification framework for affinely parametrized coercive elliptic equations was presented in \cite{Rozza2008}.
For greedy space enrichment, convergence was established in \cite{Buffa2012}, the rates were linked to the Kolmogorov \(n\)-widths of the solution manifold in \cite{Binev2011}, and the theory was extended from Hilbert spaces to Banach spaces in \cite{DeVore2013}.
For PGD, the separated approximation itself is studied through tensor approximation and functional-analytic formulations.
Within this approach, a functional version of the Eckart--Young principle characterizes separated approximations and establishes convergence properties for abstract elliptic problems \cite{FalcoNouy2011}, with the analysis subsequently extended to nonlinear convex problems in tensor Banach spaces \cite{FalcoNouy2012}.
A posteriori error estimation and adaptive enrichment strategies also address the accuracy of PGD approximations for parametrized linear parabolic equations \cite{Chamoin2017}.

A difficulty in analyzing the VS iteration is that its error need not
be orthogonal to the entire generated space in the energy inner product.
For symmetric coercive problems, Galerkin reduced approximations are
energy-orthogonal projections onto their reduced spaces and therefore
minimize the energy-norm error over those spaces \cite{Hesthaven2016}.
In contrast, the VS iteration adds each new separated term without
modifying the previously computed terms.
Although each update makes the error orthogonal to the new direction,
subsequent updates need not preserve orthogonality to earlier directions
when the bilinear form depends on the parameter.
Consequently, best-approximation estimates for the generated space
alone do not establish convergence of the VS approximations.
Since each new direction is obtained from a residual problem at a
single sampled parameter, we use cross-parameter correlations to
estimate the conditional expected energy decrease and derive
sufficient conditions for almost-sure convergence under independent
random parameter sampling.
For each VS update, we derive an exact identity for the change in the parameter-averaged error energy and show that this energy cannot increase.
A nonincreasing, nonnegative error energy may still converge to a positive limit, so we estimate its conditional expected decrease through cross-parameter correlations of the current error family.
To estimate these correlations, we project the current error family onto a finite-dimensional space, apply an average-correlation estimate to the projected errors, and bound the terms containing the approximation remainder.
In this way, we obtain a quantitative lower bound for the conditional expected decrease whenever the error energy exceeds a threshold determined by the approximation dimension and error.
We then sum the expected decreases to prove that the parameter-averaged error energy converges to zero almost surely with respect to the sampling history, provided that the evolving error family admits finite-dimensional approximations whose dimensions and errors satisfy suitable quantitative conditions.
Under further bounds on those dimensions and errors, we also obtain an inverse-logarithmic upper bound for the expected parameter-averaged error energy.
For problems with finite-term parameter-separable representations of the bilinear form and the right-hand side, we verify the abstract conditions by constructing the required approximation spaces from local inverse-operator expansions.
When the solution family is contained in a fixed finite-dimensional space, the VS errors remain in that space, so we can apply the correlation estimate directly without an approximation remainder.
In this case, the parameter-averaged error energy in \eqref{eq:error-energy-definitions} satisfies an $O(N^{-1})$ bound in expectation over the random sampling history by Theorem~\ref{thm:finite-dimensional}.
For a parameter-independent bilinear form, the iteration has an orthogonal structure in which the VS approximation is exactly the orthogonal projection of the solution onto the generated direction space in the common energy inner product, and hence satisfies a best-approximation property.
Numerical experiments illustrate these theoretical results through parameter-dependent and parameter-independent examples.

After introducing the parametric problem and the VS method in Section~\ref{sec:problem-setting}, we establish the general convergence theorem and the error bound for finite-dimensional solution families in Section~\ref{sec:convergence}.
Section~\ref{sec:finite-term} then verifies the convergence conditions for problems with separable parameter dependence, while Section~\ref{sec:fixed-operator} examines the additional projection and finite-step recovery properties available for parameter-independent operators.
Numerical experiments illustrating these results are presented in Section~\ref{sec:numerics}, followed by the concluding discussion in Section~\ref{sec:conclusion}.

\section{Problem formulation and the VS method}
\label{sec:problem-setting}

\subsection{Parametric problem and assumptions}

Consider a family of variational problems indexed by a parameter in a compact metric space $(\Omega,d_\Omega)$, equipped with its Borel $\sigma$-algebra $\mathcal B$ and a Borel probability measure $\mu$ on $(\Omega,\mathcal B)$.
The solutions are sought in a real separable Hilbert space $(\mathcal V,(\cdot,\cdot)_{\mathcal V})$ with norm $\|\cdot\|_{\mathcal V}$, and the right-hand sides are represented as elements of its dual space $\mathcal V'$.
Write $\langle\cdot,\cdot\rangle$ for the pairing between $\mathcal V'$ and $\mathcal V$, and define the dual norm of $\ell\in\mathcal V'$ by
\[
\begin{array}{ll}
\|\ell\|_{\mathcal V'}
:=
\sup_{0\ne v\in\mathcal V}
\frac{|\langle \ell,v\rangle|}{\|v\|_{\mathcal V}}.
\end{array}
\]
For each $\xi\in\Omega$, consider the variational problem of finding $u(\xi)\in\mathcal V$ such that
\begin{equation}
 a(u(\xi),v;\xi)=\langle f(\xi),v\rangle
 \qquad \forall\,v\in\mathcal V,
 \label{eq:pvp}
\end{equation}
where $a(\cdot,\cdot;\xi)$ is a bilinear form and $f(\xi)\in\mathcal V'$ is a linear functional.
Assume that the bilinear form is symmetric, so that
\begin{equation}
a(v,w;\xi)=a(w,v;\xi)
\qquad
\forall\,\xi\in\Omega,\ \forall\,v,w\in\mathcal V.
\label{eq:symmetric}
\end{equation}
In addition, its coercivity and boundedness are assumed to hold uniformly in the parameter, with constants $\alpha>0$ and $M>0$ such that
\begin{align}
 a(v,v;\xi)&\ge \alpha\|v\|_{\mathcal V}^{2}
 &&\forall\,\xi\in\Omega,\ \forall\,v\in\mathcal V,
\label{eq:coercive}
\\
 |a(v,w;\xi)|&\le M\|v\|_{\mathcal V}\|w\|_{\mathcal V}
 &&\forall\,\xi\in\Omega,\ \forall\,v,w\in\mathcal V.
 \label{eq:bounded}
\end{align}
To specify the parameter dependence, further assume that, for every fixed $v,w\in\mathcal V$, the map $\xi\mapsto a(v,w;\xi)$ is Borel measurable.
The right-hand-side map $f:\Omega\to\mathcal V'$ is likewise assumed to be Borel measurable with respect to the norm topology on $\mathcal V'$ and to satisfy the uniform bound
\begin{equation}
 F:=\sup_{\xi\in\Omega}\|f(\xi)\|_{\mathcal V'}<\infty.
 \label{eq:fbound}
\end{equation}
The assumptions on the parameter space, the Hilbert space, the bilinear form, and the right-hand side stated above remain in force throughout the paper.
By the Lax--Milgram theorem, problem \eqref{eq:pvp} has a unique solution $u(\xi)\in\mathcal V$ for every $\xi\in\Omega$.
To bound its norm uniformly in $\xi$, take $v=u(\xi)$ in \eqref{eq:pvp}, so that
\[
\begin{array}{ll}
\alpha\|u(\xi)\|_{\mathcal V}^2
\le
a(u(\xi),u(\xi);\xi)
=
\langle f(\xi),u(\xi)\rangle
\le
F\|u(\xi)\|_{\mathcal V}.
\end{array}
\]
Here the lower bound is \eqref{eq:coercive}, while the upper bound uses $|\langle f(\xi),u(\xi)\rangle|\le\|f(\xi)\|_{\mathcal V'}\|u(\xi)\|_{\mathcal V}$ and \eqref{eq:fbound}.
For $u(\xi)\ne0$, division by $\|u(\xi)\|_{\mathcal V}$ proves the uniform bound
\begin{equation}
\|u(\xi)\|_{\mathcal V}
\le
\frac{F}{\alpha}
\qquad
\forall\,\xi\in\Omega.
\label{eq:solution-uniform-bound}
\end{equation}
The bound \eqref{eq:solution-uniform-bound} also holds when $u(\xi)=0$.

For each $\xi\in\Omega$, the bilinear form $a(\cdot,\cdot;\xi)$ is symmetric by \eqref{eq:symmetric} and positive definite because $a(v,v;\xi)\ge\alpha\|v\|_{\mathcal V}^2>0$ for $v\ne0$ by \eqref{eq:coercive}.
It therefore defines an energy inner product on $\mathcal V$, with associated norm $\|v\|_{a,\xi}:=a(v,v;\xi)^{1/2}$.
The lower bound \eqref{eq:coercive} and upper bound \eqref{eq:bounded} on $a(v,v;\xi)$ make this norm uniformly equivalent to $\|\cdot\|_{\mathcal V}$, with
\[
\begin{array}{ll}
 \sqrt{\alpha}\,\|v\|_{\mathcal V}
 \le \|v\|_{a,\xi}
 \le \sqrt{M}\,\|v\|_{\mathcal V}
 \qquad
 \forall\,\xi\in\Omega,\ \forall\,v\in\mathcal V.
\end{array}
\]

The setting of \eqref{eq:pvp} includes problems in which the bilinear form
and the right-hand side admit finite-term separable representations.
This subclass is considered in Section~\ref{sec:finite-term}.

\subsection{VS iteration with independent parameter sampling}

The VS method constructs the approximation by sampling the parameter space independently according to the probability measure $\mu$ introduced in the preceding subsection.
At each sampled parameter, it solves the residual problem associated with the current approximation to obtain a spatial direction, whose contribution across the parameter domain is then determined by a coefficient function.
The construction is detailed in Algorithm~\ref{alg:VS}, where empty sums are understood to be zero.

\begin{algorithm}%
\caption{VS algorithm for parametric variational problem \eqref{eq:pvp}}
\label{alg:VS}
\begin{algorithmic}[1]
\Require the bilinear form \(a(\cdot,\cdot;\xi)\), the right-hand side \(f:\Omega\to\Vp\), and the parameter probability space \((\Omega,\mathcal B,\mu)\).

\Ensure the spatial directions \(\{g_k\}\) and the parameter-dependent coefficient functions \(\{\eta_k\}\).

\For{\(k=1,2,\dots\)}

\State
Draw \(\zeta_k\in\Omega\) independently according to \(\mu\).

\State
{\small
Find \(g_k\in\mathcal V\) such that $ a(g_k,v;\zeta_k) = \langle f(\zeta_k),v\rangle - \compactsum_{s=1}^{k-1}\eta_s(\zeta_k)a(g_s,v;\zeta_k), \ \forall\,v\in\mathcal V. $
}

\If{\(g_k=0\)}
\State Set $ \eta_k(\xi):=0 \quad \forall\,\xi\in\Omega. $
\Else
\State Set $ \eta_k(\xi) := \frac{ \langle f(\xi),g_k\rangle - \compactsum_{s=1}^{k-1}\eta_s(\xi)a(g_s,g_k;\xi) }{ a(g_k,g_k;\xi) } \quad \forall\,\xi\in\Omega. $
\EndIf

\EndFor
\end{algorithmic}
\end{algorithm}

In Algorithm \ref{alg:VS}, the approximation generated after $N$ iterations is
\begin{equation}
u_N(\xi)
=
\compactsum_{i=1}^{N}\eta_i(\xi)g_i
\qquad
\forall\,\xi\in\Omega,
\label{eq:approximate-solution}
\end{equation}
with $u_0(\xi):=0$.
The subspace spanned by $g_1,\ldots,g_N$ contains $u_N(\xi)$ for every $\xi\in\Omega$ and is denoted by
\begin{equation}
G_N:=\operatorname{span}\{g_1,\ldots,g_N\}
\quad\mbox{and}\quad G_0:=\{0\}.
\label{eq:generated-direction-space}
\end{equation}
Since each iteration adds at most one vector to $G_N$ and a zero direction adds no dimension, we have $\dim G_N\le N$.

To analyze the approximation $u_N$ in \eqref{eq:approximate-solution}, define its error at step $N$ by
\begin{equation}
R^{(N)}(\xi):=u(\xi)-u_N(\xi).
\label{eq:error-definition}
\end{equation}
The defect of $u_N(\xi)$ in the variational equation \eqref{eq:pvp} is represented by the residual functional
\begin{equation*}
\begin{array}{ll}
\mathfrak r_N(v;\xi)
:=
\langle f(\xi),v\rangle-a(u_N(\xi),v;\xi)
\qquad
\forall\,v\in\mathcal V.
\end{array}
\end{equation*}
Since $u(\xi)$ solves \eqref{eq:pvp} and $R^{(N)}(\xi)=u(\xi)-u_N(\xi)$ by \eqref{eq:error-definition}, the residual functional represents the error through
\begin{equation}
a(R^{(N)}(\xi),v;\xi)
=
\mathfrak r_N(v;\xi)
\qquad
\forall\,v\in\mathcal V.
\label{eq:error-equation}
\end{equation}
The error $R^{(N)}(\xi)$ is used only in the analysis, because Algorithm~\ref{alg:VS} computes the residual functional $\mathfrak r_N(\cdot;\xi)$ directly from $u_N(\xi)$ without requiring the exact solution $u(\xi)$.
Throughout this paper, ``error'' refers to $R^{(N)}$, whereas ``residual'' refers to the functional $\mathfrak r_N$ or to an algebraic residual in a discrete computation.

The following lemma expresses the VS updates in terms of $R^{(k-1)}$ and identifies $\eta_k(\xi)g_k$ as the orthogonal projection of $R^{(k-1)}(\xi)$ onto $\operatorname{span}\{g_k\}$ in $a(\cdot,\cdot;\xi)$.

\begin{lemma}
\label{lem:basic-error-relations}
Along any parameter sequence $\{\zeta_k\}_{k\ge1}\subset\Omega$, the direction generated by  Algorithm~\ref{alg:VS} satisfies, for every $k\ge1$,
\begin{equation}
g_k=R^{(k-1)}(\zeta_k).
\label{eq:direction-as-error}
\end{equation}
If $g_k\ne0$, then $a(g_k,g_k;\xi)>0$ for every $\xi\in\Omega$ and
\begin{equation}
\eta_k(\xi)
=
\frac{a(R^{(k-1)}(\xi),g_k;\xi)}
{a(g_k,g_k;\xi)}.
\label{eq:eta-residual}
\end{equation}
In all cases, the error $R^{(k)}$ in \eqref{eq:error-definition} satisfies, for every $\xi\in\Omega$,
\begin{equation}
R^{(k)}(\xi)
=
R^{(k-1)}(\xi)-\eta_k(\xi)g_k,
\label{eq:error-update}
\end{equation}
and
\begin{equation}
a(R^{(k)}(\xi),g_k;\xi)=0
\qquad
\forall\,\xi\in\Omega.
\label{eq:one-step-orthogonality}
\end{equation}
Moreover, the correction $\eta_k(\xi)g_k$ is the orthogonal projection of $R^{(k-1)}(\xi)$ onto $\operatorname{span}\{g_k\}$ in the inner product $a(\cdot,\cdot;\xi)$. 
\end{lemma}

\begin{proof}
Since $u_{k-1}(\xi)=\compactsum_{s=1}^{k-1}\eta_s(\xi)g_s$ in \eqref{eq:approximate-solution}, the direction equation in Algorithm~\ref{alg:VS} is $a(g_k,v;\zeta_k)=\mathfrak r_{k-1}(v;\zeta_k)$ for every $v\in\mathcal V$.
The error $R^{(k-1)}(\zeta_k)$ solves this equation by \eqref{eq:error-equation}, and the solution is unique under \eqref{eq:coercive}, so $g_k=R^{(k-1)}(\zeta_k)$ as asserted in \eqref{eq:direction-as-error}.

The correction $\eta_k(\xi)g_k$ changes the error according to \eqref{eq:error-update}, since $u_k(\xi)=u_{k-1}(\xi)+\eta_k(\xi)g_k$ in \eqref{eq:approximate-solution} and $R^{(k)}(\xi)=u(\xi)-u_k(\xi)$ by \eqref{eq:error-definition}.
For $g_k\ne0$, the denominator in Algorithm~\ref{alg:VS} satisfies $a(g_k,g_k;\xi)\ge\alpha\|g_k\|_{\mathcal V}^2>0$ by \eqref{eq:coercive}, and the numerator is $\mathfrak r_{k-1}(g_k;\xi)=a(R^{(k-1)}(\xi),g_k;\xi)$ by \eqref{eq:error-equation}.
Thus $\eta_k(\xi)$ has the form \eqref{eq:eta-residual}, for which the updated error is orthogonal to $g_k$ as in \eqref{eq:one-step-orthogonality}.
If $g_k=0$, then $\eta_k\equiv0$ in Algorithm~\ref{alg:VS}, so $u_k=u_{k-1}$ and $R^{(k)}=R^{(k-1)}$, and the orthogonality relation still holds.
In both cases, $R^{(k-1)}(\xi)=\eta_k(\xi)g_k+R^{(k)}(\xi)$ in \eqref{eq:error-update} is therefore the orthogonal decomposition relative to $\operatorname{span}\{g_k\}$, which proves the projection assertion.
\end{proof}

\section{Convergence analysis}
\label{sec:convergence}
We study convergence through the $\mu$-average over $\Omega$ of the error energy $a(R^{(N)}(\xi),R^{(N)}(\xi);\xi)$, where $R^{(N)}$ is the error in \eqref{eq:error-definition}, i.e., $R^{(N)}=u-u_N$.
Finite-dimensional approximation of $R^{(N)}$ controls the decrease of this average under one random update and leads to conditions for almost-sure convergence to zero and bounds on its expectation over the sampling history.

When $F=0$ in \eqref{eq:fbound}, the right-hand side vanishes for every parameter, so $u(\xi)=0$ by \eqref{eq:solution-uniform-bound}.
Algorithm~\ref{alg:VS} starts from $u_0=0$ and then produces only zero directions, leaving $u_N=0$. The approximation is therefore exact from the initial step, and we restrict the remaining analysis to $F>0$.

For problems with finite-term separable representations of the bilinear
form and the right-hand side, Section~\ref{sec:finite-term} verifies the
finite-dimensional approximation conditions introduced below.
Theorem~\ref{thm:finite-term-main} then shows that the VS approximations
converge almost surely to the full parametric solution in the
parameter-averaged energy norm.
We first establish the abstract criterion.

\subsection{One-step error reduction}

To formulate the random updates and their conditional expectations on a common probability space, let $\Xi_n$ denote the $n$-th coordinate map on $\Omega^{\mathbb N}$, and define
\begin{equation}
(\Theta,\mathscr A,\Pp)
:=(\Omega^{\mathbb N},\B^{\otimes\mathbb N},\mu^{\otimes\mathbb N}),
\quad
\zeta_n:=\Xi_n,
\quad\mbox{and}\quad
\F_N:=\sigma(\Xi_1,\ldots,\Xi_N).
\label{eq:sampling-framework}
\end{equation}
All random VS iterations below use the sampling model \eqref{eq:sampling-framework}.
The filtration $\F_N$ records the samples available through iteration $N$, with $\F_0:=\{\emptyset,\Theta\}$.
In Algorithm~\ref{alg:VS}, the error family $R^{(N)}$ in \eqref{eq:error-definition}, i.e., $R^{(N)}=u-u_N$, is constructed from $\Xi_1,\ldots,\Xi_N$, including iterations in which a zero direction produces an identity update.
Thus conditioning on $\F_N$ retains the information that determines the current error, while the next sample $\Xi_{N+1}$ remains independent of $\F_N$ with distribution $\mu$ under \eqref{eq:sampling-framework}.
Write $\omega\in\Theta$ for a sampling history and suppress this dependence in the iterates and error quantities whenever no confusion can arise.
Almost-sure statements and expectations for the VS iteration refer to the sampling measure $\Pp$ on $\Theta$, whereas almost-everywhere statements in the parameter variable refer to the parameter measure $\mu$ on $\Omega$.

Define the pointwise error energy and the parameter-averaged error energy by
\begin{equation}
 E_N(\xi):=a(R^{(N)}(\xi),R^{(N)}(\xi);\xi)
 \quad\mbox{and}\quad
 \cE_N:=\int_{\Omega}E_N(\xi)\dd\mu(\xi).
 \label{eq:error-energy-definitions}
\end{equation}

To estimate the expected decrease of $\cE_N$, define the squared cross-parameter correlation of the error family $R^{(N)}$ by
\begin{equation}
J_N:=\int_{\Omega}\int_{\Omega}
 a(R^{(N)}(\xi),R^{(N)}(\zeta);\xi)^2
 \dd\mu(\zeta)\dd\mu(\xi).
\label{eq:correlation-functional}
\end{equation}

\begin{proposition}
\label{prop:energy-identity}
Let $\{u_N\}_{N\ge0}$ be generated by Algorithm~\ref{alg:VS}, and let $R^{(N)}$, $E_N$, $\cE_N$, and $J_N$ be defined by \eqref{eq:error-definition}, \eqref{eq:error-energy-definitions}, and \eqref{eq:correlation-functional}, respectively.
Fix a sampling history through step $N$ and a candidate value $\zeta\in\Omega$ for the next sample, and perform one further update of Algorithm~\ref{alg:VS} with $\zeta_{N+1}=\zeta$.
Then, for every $\xi\in\Omega$,
\begin{equation}
 E_N(\xi)-E_{N+1}(\xi)
 =
 \tfrac{a(R^{(N)}(\xi),R^{(N)}(\zeta);\xi)^2}
 {a(R^{(N)}(\zeta),R^{(N)}(\zeta);\xi)},
 \label{eq:energy-identity}
\end{equation}
where the right-hand side is defined to be zero when $R^{(N)}(\zeta)=0$.
In particular, $E_{N+1}(\xi)\le E_N(\xi)$ for every $\xi\in\Omega$ and $\cE_{N+1}\le\cE_N$.

For the independent random sample $\zeta_{N+1}=\Xi_{N+1}$, the conditional expected decrease satisfies, $\Pp$-almost surely,
\begin{equation}
\E[\cE_N-\cE_{N+1}\mid\F_N]
=\int_{\Omega}\int_{\Omega}
\frac{a(R^{(N)}(\xi),R^{(N)}(\zeta);\xi)^2}
{a(R^{(N)}(\zeta),R^{(N)}(\zeta);\xi)}
\dd\mu(\zeta)\dd\mu(\xi),
\label{eq:conditional-exact}
\end{equation}
where the integrand is defined to be zero when $R^{(N)}(\zeta)=0$.
Moreover, the conditional expected decrease satisfies, $\Pp$-almost surely,
\begin{equation}
\E[\cE_N-\cE_{N+1}\mid\F_N]
\ge \frac{\alpha^2}{MF^2}J_N.
\label{eq:conditional-correlation}
\end{equation}
\end{proposition}

\begin{proof}
By \eqref{eq:direction-as-error} in Lemma~\ref{lem:basic-error-relations}, $g_{N+1}=R^{(N)}(\zeta)$.
For $g_{N+1}\ne0$, the denominator satisfies $a(g_{N+1},g_{N+1};\xi)\ge\alpha\|g_{N+1}\|_{\mathcal V}^2>0$ by \eqref{eq:coercive}, so the coefficient in \eqref{eq:eta-residual} is well defined.
The error update \eqref{eq:error-update} expands $E_{N+1}(\xi)$ in \eqref{eq:error-energy-definitions} into a quadratic expression whose mixed terms $a(R^{(N)}(\xi),g_{N+1};\xi)$ and $a(g_{N+1},R^{(N)}(\xi);\xi)$ coincide by \eqref{eq:symmetric}.
With the coefficient \eqref{eq:eta-residual}, the expanded energy is
\[
\begin{array}{ll}
E_{N+1}(\xi)
&=E_N(\xi)-2\eta_{N+1}(\xi)a(R^{(N)}(\xi),g_{N+1};\xi)
  +\eta_{N+1}(\xi)^2a(g_{N+1},g_{N+1};\xi)\\[2mm]
&=E_N(\xi)
  -\frac{a(R^{(N)}(\xi),g_{N+1};\xi)^2}
  {a(g_{N+1},g_{N+1};\xi)}.
\end{array}
\]
Because $g_{N+1}=R^{(N)}(\zeta)$, the difference $E_N(\xi)-E_{N+1}(\xi)$ is exactly \eqref{eq:energy-identity} in the nonzero case.
If $g_{N+1}=0$, the update leaves $R^{(N+1)}(\xi)=R^{(N)}(\xi)$, so both the energy decrease and the numerator in \eqref{eq:energy-identity} vanish, in agreement with the zero value assigned to the quotient in \eqref{eq:energy-identity}.
Thus the right-hand side of \eqref{eq:energy-identity} is nonnegative in both cases, so $E_{N+1}(\xi)\le E_N(\xi)$ and, after integration with respect to $\xi$, $\cE_{N+1}\le\cE_N$.

To estimate the conditional expected decrease, first bound the initial energy using the test function $v=u(\xi)$ in \eqref{eq:pvp}.
Since $\|f(\xi)\|_{\V'}\le F$ by \eqref{eq:fbound} and $\|u(\xi)\|_{\V}\le F/\alpha$ by \eqref{eq:solution-uniform-bound},
\[
\begin{array}{ll}
E_0(\xi)
=a(u(\xi),u(\xi);\xi)
=\langle f(\xi),u(\xi)\rangle
\le F\|u(\xi)\|_{\V}
\le\frac{F^2}{\alpha}.
\end{array}
\]
The monotonicity $E_N(\xi)\le E_0(\xi)$ extends the bound $F^2/\alpha$ to every step, while $\alpha\|R^{(N)}(\xi)\|_{\V}^2\le E_N(\xi)$ by \eqref{eq:coercive}.
Hence the error is bounded uniformly over the parameters and the sampling history by
\begin{equation}
\|R^{(N)}(\xi)\|_{\V}\le \tfrac{F}{\alpha}
\qquad \forall\,\xi\in\Omega.
\label{eq:error-uniform-bound}
\end{equation}

For a fixed sampling history through step $N$ and a candidate parameter $\zeta$ for the next update, the integral of \eqref{eq:energy-identity} with respect to $\xi$ is
\begin{equation}
\begin{array}{ll}
\cE_N-\cE_{N+1}
=
\int_{\Omega}
\frac{a(R^{(N)}(\xi),R^{(N)}(\zeta);\xi)^2}
     {a(R^{(N)}(\zeta),R^{(N)}(\zeta);\xi)}
\dd\mu(\xi),
\end{array}
\label{eq:candidate-energy-decrease}
\end{equation}
where $\cE_{N+1}$ is the parameter-averaged error energy after the update at $\zeta$.
By Proposition~\ref{prop:iteration-measurability} in Appendix~\ref{app:measurability}, $R^{(N)}$ is strongly measurable with respect to $\mathcal B\otimes\F_N$.
The map $(\xi,v,w)\mapsto a(v,w;\xi)$ is jointly measurable by Lemma~\ref{lem:solution-measurability}, so the functions $a(R^{(N)}(\xi),R^{(N)}(\zeta);\xi)$ and $a(R^{(N)}(\zeta),R^{(N)}(\zeta);\xi)$ are jointly measurable in $\xi$, $\zeta$, and the history through step $N$.
The quotient in \eqref{eq:candidate-energy-decrease} is therefore measurable on $\{R^{(N)}(\zeta)\ne0\}$, where its denominator is positive, and its extension by zero to $\{R^{(N)}(\zeta)=0\}$ is jointly measurable as well.
The quotient in \eqref{eq:candidate-energy-decrease} lies between zero and $E_N(\xi)\le F^2/\alpha$ by the Cauchy--Schwarz inequality for $R^{(N)}(\xi)$ and $R^{(N)}(\zeta)$ in $a(\cdot,\cdot;\xi)$, so it is integrable over both parameters and the sampling history.
Since Algorithm~\ref{alg:VS} uses $\zeta=\Xi_{N+1}$ with conditional distribution $\mu$ given $\F_N$, the conditional expectation of $\cE_N-\cE_{N+1}$ given $\F_N$ equals the integral of the right-hand side of \eqref{eq:candidate-energy-decrease} with respect to $\mu(\zeta)$.
The two parameter integrals may be interchanged by Fubini's theorem, proving \eqref{eq:conditional-exact}.

For $R^{(N)}(\zeta)\ne0$, the denominator in \eqref{eq:conditional-exact} satisfies, by \eqref{eq:bounded} and \eqref{eq:error-uniform-bound},
\[
\begin{array}{ll}
a(R^{(N)}(\zeta),R^{(N)}(\zeta);\xi)
\le M\|R^{(N)}(\zeta)\|_{\V}^2
\le\frac{MF^2}{\alpha^2}.
\end{array}
\]
Thus, on the set where $R^{(N)}(\zeta)\ne0$, the reciprocal of the denominator in \eqref{eq:conditional-exact} is at least $\alpha^2/(MF^2)$.
On $\{R^{(N)}(\zeta)=0\}$, both the integrand in \eqref{eq:conditional-exact} and $a(R^{(N)}(\xi),R^{(N)}(\zeta);\xi)^2$ vanish.
The integral in \eqref{eq:conditional-exact} is therefore at least $\alpha^2J_N/(MF^2)$ by the definition of $J_N$ in \eqref{eq:correlation-functional}, proving \eqref{eq:conditional-correlation}.
\end{proof}

\subsection{A finite-dimensional correlation estimate}

The conditional decrease bound \eqref{eq:conditional-correlation} in Proposition~\ref{prop:energy-identity} reduces the analysis to finding a lower bound for the correlation functional $J_N$ in \eqref{eq:correlation-functional} in terms of the parameter-averaged error energy $\cE_N$ in \eqref{eq:error-energy-definitions}.
When $R^{(N)}$ takes values in a finite-dimensional subspace of $\V$, the following lemma bounds $J_N$ from below by $\cE_N^2$ divided by a positive upper bound for the subspace dimension.
We formulate the lemma for a general square-integrable function so that it also applies to finite-dimensional approximations of $R^{(N)}$.

\begin{lemma}
\label{lem:finite-correlation}
Let $a(\cdot,\cdot;\xi)$ be the bilinear form in the variational problem \eqref{eq:pvp}.
Let $d_Z\ge1$ be an integer and let $Z\subset\V$ be a subspace with $\dim Z\le d_Z$.
If $X:\Omega\to Z$ is strongly measurable and belongs to $L^2(\Omega,\mu;Z)$, then
\begin{equation}
\int_{\Omega}\int_{\Omega}a(X(\xi),X(\zeta);\xi)^2
\dd\mu(\zeta)\dd\mu(\xi)
\ge\frac1{d_Z}
\Bigl[\int_{\Omega}a(X(\xi),X(\xi);\xi)\dd\mu(\xi)\Bigr]^2.
\label{eq:finite-correlation}
\end{equation}
\end{lemma}

\begin{proof}
When $Z=\{0\}$, $X=0$ and both sides of \eqref{eq:finite-correlation} vanish, so consider $1\le s=\dim Z\le d_Z$ and choose a $\V$-orthonormal basis $e_1,\ldots,e_s$.
In this basis, write $X(\xi)=\compactsum_{j=1}^{s}x_j(\xi)e_j$, with $x(\xi)=(x_1(\xi),\ldots,x_s(\xi))^T$, and define the matrix $B(\xi)$ by $B_{ij}(\xi)=a(e_i,e_j;\xi)$ for $1\le i,j\le s$.
Since $B_{ij}(\xi)=B_{ji}(\xi)$ by \eqref{eq:symmetric}, we have $B(\xi)^T=B(\xi)$.
The positive semidefinite second-moment matrix $C:=\int_{\Omega}x(\zeta)x(\zeta)^T\dd\mu(\zeta)$ records the correlations across parameters. It need not be invertible, so let $C^\dagger$ denote its Moore--Penrose pseudoinverse.
For each $y\in\ker C$, the identity $\int_\Omega |y^Tx(\zeta)|^2\dd\mu(\zeta)=0$ means that $y^Tx(\zeta)=0$ for $\mu$-almost every $\zeta$.
The intersection of these full-measure sets over a finite basis of $\ker C$ still has full measure. On this set, $x(\zeta)\in(\ker C)^\perp=\operatorname{range}C$, where the equality uses the symmetry of $C$.
Hence the orthogonal projection $P_C$ onto $\operatorname{range}C$ satisfies $P_Cx(\xi)=x(\xi)$ for $\mu$-almost every $\xi$.
For fixed $\xi$, the definition of $C$ gives
\begin{equation}
\int_{\Omega}a(X(\xi),X(\zeta);\xi)^2\dd\mu(\zeta)=x(\xi)^TB(\xi)CB(\xi)x(\xi).
\label{eq:correlation-matrix-identity}
\end{equation}
Set $y(\xi):=C^{1/2}B(\xi)x(\xi)$ and $z(\xi):=(C^\dagger)^{1/2}x(\xi)$.
Their inner product is $y(\xi)^Tz(\xi)=x(\xi)^TB(\xi)x(\xi)=a(X(\xi),X(\xi);\xi)$ for $\mu$-almost every $\xi$, since $C^{1/2}(C^\dagger)^{1/2}=P_C$ and $P_Cx(\xi)=x(\xi)$.
The Cauchy--Schwarz inequality for these vectors $y(\xi)$ and $z(\xi)$ in $\mathbb R^s$ now reads
\begin{equation}
\begin{array}{ll}
\bigl(x(\xi)^TB(\xi)x(\xi)\bigr)^2
\le x(\xi)^TB(\xi)CB(\xi)x(\xi)\,
 x(\xi)^TC^\dagger x(\xi).
\end{array}
\label{eq:correlation-pointwise-cs}
\end{equation}

The bound $x(\xi)^TB(\xi)x(\xi)=a(X(\xi),X(\xi);\xi)\ge\alpha\|X(\xi)\|_{\V}^2\ge0$ from \eqref{eq:coercive} permits us to take square roots in \eqref{eq:correlation-pointwise-cs}.
We then integrate over $\Omega$ and apply the Cauchy--Schwarz inequality to the scalar functions $\xi\mapsto\bigl(x(\xi)^TB(\xi)CB(\xi)x(\xi)\bigr)^{1/2}$ and $\xi\mapsto\bigl(x(\xi)^TC^\dagger x(\xi)\bigr)^{1/2}$ in $L^2(\Omega,\mu)$, which leads to
\begin{equation}
\begin{array}{ll}
\bigl[\int x(\xi)^TB(\xi)x(\xi)\dd\mu(\xi)\bigr]^2
&\le
\bigl[\int x(\xi)^TB(\xi)CB(\xi)x(\xi)\dd\mu(\xi)\bigr]
\\[2mm]
&\quad\times\bigl[\int x(\xi)^TC^\dagger x(\xi)\dd\mu(\xi)\bigr].
\end{array}
\label{eq:correlation-integrated-cs}
\end{equation}
The second integral on the right-hand side of \eqref{eq:correlation-integrated-cs} equals $\int_\Omega x(\xi)^TC^\dagger x(\xi)\dd\mu(\xi)=\tr(C^\dagger C)=\rank C\le s\le d_Z$.
With this bound, \eqref{eq:correlation-integrated-cs} becomes \eqref{eq:finite-correlation} by \eqref{eq:correlation-matrix-identity} and $x(\xi)^TB(\xi)x(\xi)=a(X(\xi),X(\xi);\xi)$.
\end{proof}

\subsection{Expected error reduction under finite-dimensional approximation}

The error $R^{(N)}=u-u_N$ in \eqref{eq:error-definition} need not take values in a finite-dimensional subspace, so we approximate it by its $\V$-orthogonal projection onto such a subspace.
Lemma~\ref{lem:finite-correlation} bounds the projected correlation from below. We estimate the contribution of the remaining error to recover a lower bound for $J_N$, provided that the parameter-averaged error energy $\cE_N$ in \eqref{eq:error-energy-definitions} is sufficiently large relative to the subspace dimension and the projection error.
Proposition~\ref{prop:tail-or-decrease} states this bound and its consequence for $\E[\cE_N-\cE_{N+1}\mid\F_N]$ through \eqref{eq:conditional-correlation} in Proposition~\ref{prop:energy-identity}.

To state the approximation assumptions, define $\operatorname{dist}_{\mathcal V}(v,W):=\inf_{w\in W}\|v-w\|_{\mathcal V}$ for a closed subspace $W\subset\mathcal V$.
For each $N$, suppose that the error family can be approximated in an $\F_N$-dependent finite-dimensional subspace $Z_N\subset\mathcal V$, whose $\mathcal V$-orthogonal projection $P_N:\mathcal V\to Z_N$ satisfies
\begin{equation}
\omega\mapsto P_N(\omega)v
\quad\text{is }\F_N\text{-measurable for every fixed }v\in\mathcal V.
\label{eq:projection-measurability}
\end{equation}
Assume that there are deterministic positive integers $D_N$ and deterministic nonnegative numbers $\tau_N$ such that, for each $N$ and almost every sampling history,
\begin{equation}
\dim Z_N\le D_N
\quad\mbox{and}\quad
\operatorname*{ess\,sup}_{\xi\in\Omega}
\dist_{\V}(R^{(N)}(\xi),Z_N)\le\tau_N.
\label{eq:abstract-approx}
\end{equation}
Here $D_N$ bounds the dimension of the approximation space $Z_N$ and need not equal the number of directions generated by Algorithm~\ref{alg:VS}.
Because there are countably many iterations, we may intersect these full-measure events and work on a single event of probability one where \eqref{eq:abstract-approx} holds for every $N$.
On this event, the orthogonal projection separates the error into a component in $Z_N$ and an approximation tail through the decomposition
\begin{equation}
R^{(N)}(\xi)=X_N(\xi)+Y_N(\xi),\quad X_N=P_NR^{(N)}\quad\mbox{and}\quad Y_N=(I-P_N)R^{(N)}.
\label{eq:XY-decomposition}
\end{equation}
The projection property gives $\|Y_N(\xi)\|_{\V}=\dist_{\V}(R^{(N)}(\xi),Z_N)\le\tau_N$ for $\mu$-almost every $\xi$ and places $X_N(\xi)$ in $Z_N$.
For each sampling history, the orthogonal projection $P_N$ is continuous in $v$, while for each fixed $v$ the map $\omega\mapsto P_N(\omega)v$ is $\F_N$-measurable by \eqref{eq:projection-measurability}.
The Carath\'eodory joint measurability theorem therefore applies on the separable space $\mathcal V$, so $(v,\omega)\mapsto P_N(\omega)v$ is jointly measurable.
By Proposition~\ref{prop:iteration-measurability} in Appendix~\ref{app:measurability}, the map $(\xi,\omega)\mapsto(R^{(N)}(\xi,\omega),\omega)$ is measurable.
Composing this map with $(v,\omega)\mapsto P_N(\omega)v$ proves the joint strong measurability of $X_N=P_NR^{(N)}$.
Since subtraction is continuous on $\V\times\V$, $Y_N=R^{(N)}-X_N$ is jointly strongly measurable as well.

\begin{proposition}
\label{prop:tail-or-decrease}
Let $\{u_N\}_{N\ge0}$ be generated by Algorithm~\ref{alg:VS}.
Let $R^{(N)}$, $\cE_N$, and $J_N$ be defined by \eqref{eq:error-definition}, \eqref{eq:error-energy-definitions}, and \eqref{eq:correlation-functional}, respectively.
Assume that the finite-dimensional spaces $Z_N$ satisfy \eqref{eq:abstract-approx} for every $N$ and $\Pp$-almost every sampling history, with deterministic positive integers $D_N$ and deterministic nonnegative numbers $\tau_N$, and that the $\mathcal V$-orthogonal projections $P_N$ onto $Z_N$ satisfy \eqref{eq:projection-measurability}.
There are constants $C_0,c_0>0$, depending only on $\alpha$ and $M$, such that, for every $N$ and $\Pp$-almost every sampling history satisfying
\begin{equation}
\cE_N\ge C_0D_N\tau_N^2,
\label{eq:above-tail}
\end{equation}
the correlation functional and the conditional expected decrease obey
\[
\begin{array}{ll}
J_N\ge c_0\dfrac{\cE_N^2}{D_N},
\qquad
\E[\cE_N-\cE_{N+1}\mid\F_N]
\ge c_*\dfrac{\cE_N^2}{D_N},
\end{array}
\]
where $c_*:=\alpha^2c_0/(MF^2)$ and $F$ is defined in \eqref{eq:fbound}. 
\end{proposition}

\begin{proof}
Set $S_N:=\int_{\Omega}\|X_N(\xi)\|_{\mathcal V}^2\dd\mu(\xi)$ to measure the projected error.
The inclusion $X_N(\xi)\in Z_N$ and the dimension bound $\dim Z_N\le D_N$ allow us to apply Lemma~\ref{lem:finite-correlation} to $X_N$ in the form
\[
\begin{array}{ll}
\int_\Omega\int_\Omega
a(X_N(\xi),X_N(\zeta);\xi)^2
\dd\mu(\zeta)\dd\mu(\xi)\ge
\frac1{D_N}
\Bigl[
\int_\Omega a(X_N(\xi),X_N(\xi);\xi)\dd\mu(\xi)
\Bigr]^2.
\end{array}
\]
Taking square roots and using the uniform coercivity bound $a(X_N(\xi),X_N(\xi);\xi)\ge\alpha\|X_N(\xi)\|_{\mathcal V}^2$, we obtain
\begin{equation}
\bigl\|a(X_N(\xi),X_N(\zeta);\xi)\bigr\|_{L^2(\mu\otimes\mu)}
\ge
\frac{\alpha}{\sqrt{D_N}}S_N.
\label{eq:main-correlation-norm}
\end{equation}
To express \eqref{eq:main-correlation-norm} in terms of $\cE_N$, we need a lower bound on the squared $L^2(\Omega,\mu;\V)$ norm $S_N$ of the projected error.
The decomposition \eqref{eq:XY-decomposition} and \eqref{eq:bounded} bound $E_N(\xi)$ by $M(\|X_N(\xi)\|_{\V}+\|Y_N(\xi)\|_{\V})^2$.
The inequality $(s+t)^2\le2(s^2+t^2)$ separates these contributions, while $\mu(\Omega)=1$ and the almost-everywhere bound $\|Y_N(\xi)\|_{\mathcal V}\le\tau_N$ limit the integrated squared tail norm to $\tau_N^2$. Consequently,
\[
\begin{array}{ll}
\cE_N
&\le M\int_\Omega
\bigl(\|X_N(\xi)\|_{\mathcal V}
      +\|Y_N(\xi)\|_{\mathcal V}\bigr)^2
\dd\mu(\xi)\\[2mm]
&\le2M\int_\Omega
\bigl(\|X_N(\xi)\|_{\mathcal V}^2
      +\|Y_N(\xi)\|_{\mathcal V}^2\bigr)
\dd\mu(\xi)\le2M(S_N+\tau_N^2).
\end{array}
\]
Under the threshold condition \eqref{eq:above-tail}, $\tau_N^2\le\cE_N/(C_0D_N)\le\cE_N/C_0$ because $D_N\ge1$.
If $C_0\ge4M$, then $\tau_N^2\le\cE_N/(4M)$, and we can isolate $S_N$ in $\cE_N\le2M(S_N+\tau_N^2)$ as
\begin{equation}
S_N
\ge\frac{\cE_N}{2M}-\tau_N^2
\ge\frac{\cE_N}{2M}-\frac{\cE_N}{4M}
=\frac{\cE_N}{4M}.
\label{eq:SN-lower}
\end{equation}

The decomposition \eqref{eq:XY-decomposition} relates the projected correlation in \eqref{eq:main-correlation-norm} to the full correlation functional $J_N$ in \eqref{eq:correlation-functional} through the expansion
\[
\begin{array}{ll}
a(R^{(N)}(\xi),R^{(N)}(\zeta);\xi)
={}&a(X_N(\xi),X_N(\zeta);\xi)+a(X_N(\xi),Y_N(\zeta);\xi)\\[2mm]
&+a(Y_N(\xi),X_N(\zeta);\xi)+a(Y_N(\xi),Y_N(\zeta);\xi).
\end{array}
\]
The three terms containing $Y_N$ account for the approximation tail, which we control in $L^2(\mu\otimes\mu)$.
For the first mixed term, we square the bound \eqref{eq:bounded} with $v=X_N(\xi)$ and $w=Y_N(\zeta)$ and separate the parameter integrals by Fubini's theorem.
The integral $\int_\Omega\|X_N(\xi)\|_{\V}^2\dd\mu(\xi)$ equals $S_N$, while $\int_\Omega\|Y_N(\zeta)\|_{\V}^2\dd\mu(\zeta)\le\tau_N^2$, so
\[
\begin{array}{ll}
\bigl\|a(X_N(\xi),Y_N(\zeta);\xi)\bigr\|_{L^2(\mu\otimes\mu)}^2
&\le M^2\int_\Omega\int_\Omega
\|X_N(\xi)\|_{\mathcal V}^2
\|Y_N(\zeta)\|_{\mathcal V}^2
\dd\mu(\zeta)\dd\mu(\xi)\\[2mm]
&\le M^2S_N\tau_N^2.
\end{array}
\]
Thus $a(X_N(\xi),Y_N(\zeta);\xi)$ has $L^2(\mu\otimes\mu)$-norm at most $M\sqrt{S_N}\tau_N$.
The same calculation, with $\|Y_N(\xi)\|_{\mathcal V}\le\tau_N$, bounds the norms of $a(Y_N(\xi),X_N(\zeta);\xi)$ and $a(Y_N(\xi),Y_N(\zeta);\xi)$ by $M\sqrt{S_N}\tau_N$ and $M\tau_N^2$, respectively.
We insert these three estimates and \eqref{eq:main-correlation-norm} into the reverse triangle inequality in $L^2(\mu\otimes\mu)$ to compare the full and projected correlations as
\begin{equation}
\sqrt{J_N}
\ge\frac{\alpha}{\sqrt{D_N}}S_N
-2M\sqrt{S_N}\tau_N-M\tau_N^2.
\label{eq:tail-subtraction}
\end{equation}

It remains to show that the two tail terms $2M\sqrt{S_N}\tau_N$ and $M\tau_N^2$ in \eqref{eq:tail-subtraction} are small relative to $\alpha S_N/\sqrt{D_N}$.
If $\cE_N=0$, the desired bound $J_N\ge c_0\cE_N^2/D_N$ follows from $J_N\ge0$.
For $\cE_N>0$, the bound $S_N\ge\cE_N/(4M)>0$ in \eqref{eq:SN-lower} permits division by $\alpha S_N/\sqrt{D_N}$. Together with the threshold condition $\cE_N\ge C_0D_N\tau_N^2$, it controls the ratio
\[
\begin{array}{ll}
\frac{D_N\tau_N^2}{S_N}
\le
\frac{\cE_N/C_0}{\cE_N/(4M)}
=\frac{4M}{C_0}.
\end{array}
\]
Consequently,
\[
\begin{cases}
\frac{2M\sqrt{S_N}\tau_N}{\alpha S_N/\sqrt{D_N}}
=\frac{2M}{\alpha}\sqrt{\frac{D_N\tau_N^2}{S_N}}
\le\frac{4M\sqrt M}{\alpha\sqrt{C_0}},\\[2mm]
\frac{M\tau_N^2}{\alpha S_N/\sqrt{D_N}}
=\frac{M}{\alpha}\frac{D_N\tau_N^2}{S_N\sqrt{D_N}}
\le\frac{M}{\alpha}\frac{D_N\tau_N^2}{S_N}
\le\frac{4M^2}{\alpha C_0}.
\end{cases}
\]
Choose $C_0$ sufficiently large that $4M\sqrt M/(\alpha\sqrt{C_0})+4M^2/(\alpha C_0)\le1/2$.
With this choice of $C_0$, the sum $2M\sqrt{S_N}\tau_N+M\tau_N^2$ is at most $\alpha S_N/(2\sqrt{D_N})$, and the correlation functional satisfies, by \eqref{eq:tail-subtraction} and \eqref{eq:SN-lower},
\[
\begin{array}{ll}
\sqrt{J_N}\ge\frac{\alpha}{2\sqrt{D_N}}S_N
\ge\frac{\alpha}{2\sqrt{D_N}}\frac{\cE_N}{4M}
=\frac{\alpha}{8M\sqrt{D_N}}\cE_N.
\end{array}
\]
Squaring $\sqrt{J_N}\ge\alpha\cE_N/(8M\sqrt{D_N})$ yields
\begin{equation}
\begin{array}{ll}
J_N\ge
\frac{\alpha^2}{64M^2}
\frac{\cE_N^2}{D_N}.
\end{array}
\label{eq:full-correlation-lower}
\end{equation}
Thus $c_0=\alpha^2/(64M^2)$ is admissible.

Since $J_N\ge c_0\cE_N^2/D_N$ by \eqref{eq:full-correlation-lower}, with $c_0=\alpha^2/(64M^2)$, the conditional expected decrease satisfies the estimate \eqref{eq:conditional-correlation} from Proposition~\ref{prop:energy-identity} in the form
\[
\begin{array}{ll}
\E[\cE_N-\cE_{N+1}\mid\F_N]
\ge\frac{\alpha^2}{MF^2}J_N
\ge\frac{\alpha^2c_0}{MF^2}\frac{\cE_N^2}{D_N}
=c_*\frac{\cE_N^2}{D_N}.
\end{array}
\]
An explicit admissible choice of $C_0$ is given in Appendix~\ref{app:tail}.
\end{proof}

\subsection{Almost-sure convergence and expected error bounds}

We now use Proposition~\ref{prop:tail-or-decrease} to establish conditions under which the parameter-averaged error energy $\cE_N$ in \eqref{eq:error-energy-definitions} converges to zero $\Pp$-almost surely.
For the dimension bounds $D_N$ and tail tolerances $\tau_N$ in \eqref{eq:abstract-approx}, assume that $D_N\tau_N^2\to0$.
The threshold $C_0D_N\tau_N^2$ in \eqref{eq:above-tail} then tends to zero, so Proposition~\ref{prop:tail-or-decrease} applies whenever $\cE_N$ stays above a fixed positive value and $N$ is sufficiently large.
To rule out a positive limit of $\cE_N$, the proof sums the lower bounds for the expected decreases and uses $\compactsum_{N=0}^{\infty}D_N^{-1}=\infty$ to contradict the upper bound $\cE_0$ on the expected total decrease.
Under additional quantitative bounds on $D_N$ and $\tau_N$, the following theorem also bounds $\E[\cE_N]$, where the expectation is taken over the sampling history.

\begin{theorem}
\label{thm:abstract-main}
Let $\{u_N\}_{N\ge0}$ be generated by Algorithm~\ref{alg:VS}, and let $R^{(N)}$ and $\cE_N$ be defined by \eqref{eq:error-definition} and \eqref{eq:error-energy-definitions}.
Suppose that there are $\F_N$-dependent finite-dimensional spaces $Z_N$ with $\mathcal V$-orthogonal projections $P_N$ satisfying \eqref{eq:projection-measurability}, together with deterministic positive integers $D_N$ and deterministic nonnegative numbers $\tau_N$ satisfying \eqref{eq:abstract-approx}, and that
\begin{equation}
\begin{array}{ll}
D_N\tau_N^2\to0
\quad\mbox{and}\quad
\compactsum_{N=0}^{\infty}\frac1{D_N}=\infty.
\end{array}
\label{eq:abstract-convergence-assumptions}
\end{equation}
Then
\begin{equation*}
\begin{array}{ll}
\cE_N\to0
\qquad \Pp\text{-a.s.}
\end{array}
\end{equation*}
Let $C_0$ be the threshold constant in Proposition~\ref{prop:tail-or-decrease}.
Suppose in addition that there are constants $B,C_{\gamma}>0$ such that, for all sufficiently large $N$,
\begin{equation}
C_0D_N\tau_N^2\le\frac{B}{N+1}
\quad\mbox{and}\quad
D_N\le C_{\gamma}(N+1).
\label{eq:abstract-rate-assumptions}
\end{equation}
Then there is a constant $C>0$ such that
\begin{equation*}
\begin{array}{ll}
\E[\cE_N]\le\frac{C}{\log(N+2)}
\end{array}
\end{equation*}
for all sufficiently large $N$.
\end{theorem}

\begin{proof}
By Proposition~\ref{prop:energy-identity}, $0\le\cE_{N+1}\le\cE_N$ along every sampling history, so $\cE_N$ converges monotonically to a nonnegative random limit $\cE_*$.
To exclude a positive limit, fix $\varepsilon>0$ and consider the sampling histories in $A_N^\varepsilon:=\{\cE_N\ge\varepsilon\}$.
By Proposition~\ref{prop:iteration-measurability}, $\cE_N$ is $\F_N$-measurable, so $A_N^\varepsilon=\{\omega\in\Theta:\cE_N(\omega)\ge\varepsilon\}$ belongs to $\F_N$.
Since $D_N\tau_N^2\to0$, choose $N_\varepsilon$ so that $C_0D_N\tau_N^2<\varepsilon$ for every $N\ge N_\varepsilon$.
On $A_N^\varepsilon$ and for $N\ge N_\varepsilon$, the energy satisfies $\cE_N\ge\varepsilon>C_0D_N\tau_N^2$, so the conditional decrease estimate in Proposition~\ref{prop:tail-or-decrease} applies in the form
\[
\begin{array}{ll}
\mathbf1_{A_N^\varepsilon}
\E[\cE_N-\cE_{N+1}\mid\F_N]
\ge
c_*\frac{\varepsilon^2}{D_N}
\mathbf1_{A_N^\varepsilon}.
\end{array}
\]
Since $\mathbf1_{A_N^\varepsilon}$ is $\F_N$-measurable, the expected decrease on $A_N^\varepsilon$ satisfies
\begin{equation}
\begin{array}{ll}
\E[(\cE_N-\cE_{N+1})\mathbf1_{A_N^\varepsilon}]
=\E\!\bigl[
\mathbf1_{A_N^\varepsilon}
\E[\cE_N-\cE_{N+1}\mid\F_N]
\bigr]
\ge c_*\frac{\varepsilon^2}{D_N}\Pp(A_N^\varepsilon).
\end{array}
\label{eq:expected-decrease-on-event}
\end{equation}
To bound $\compactsum_{N=N_\varepsilon}^{K}\Pp(A_N^\varepsilon)/D_N$, sum \eqref{eq:expected-decrease-on-event} from $N=N_\varepsilon$ to $K$.
Each decrease $\cE_N-\cE_{N+1}$ is nonnegative, so the indicators may be bounded by $1$, after which the differences telescope to $\cE_{N_\varepsilon}-\cE_{K+1}$ and
\[
\begin{array}{ll}
c_*\varepsilon^2\compactsum_{N=N_\varepsilon}^{K}
\frac{\Pp(A_N^\varepsilon)}{D_N}
&\le\compactsum_{N=N_\varepsilon}^{K}
\E[(\cE_N-\cE_{N+1})\mathbf1_{A_N^\varepsilon}]\\[2mm]
&\le\E\bigl[\compactsum_{N=N_\varepsilon}^{K}(\cE_N-\cE_{N+1})\bigr]
=\E[\cE_{N_\varepsilon}-\cE_{K+1}]
\le\E[\cE_{N_\varepsilon}].
\end{array}
\]
The upper bound $\E[\cE_{N_\varepsilon}]\le\cE_0<\infty$ is independent of $K$, so letting $K\to\infty$ yields
\begin{equation}
\compactsum_{N=N_\varepsilon}^{\infty}
\frac{\Pp(A_N^\varepsilon)}{D_N}<\infty.
\label{eq:weighted-probability-sum}
\end{equation}
If $\Pp(\cE_*\ge\varepsilon)>0$, the inequality $\cE_N\ge\cE_*$ places $\{\cE_*\ge\varepsilon\}$ inside $A_N^\varepsilon$ for every $N$, so
\[
\begin{array}{ll}
\compactsum_{N=N_\varepsilon}^{\infty}
\frac{\Pp(A_N^\varepsilon)}{D_N}
\ge\Pp(\cE_*\ge\varepsilon)
\compactsum_{N=N_\varepsilon}^{\infty}\frac1{D_N}
=\infty,
\end{array}
\]
which contradicts \eqref{eq:weighted-probability-sum} and forces $\Pp(\cE_*\ge\varepsilon)=0$.
In particular, $\Pp(\cE_*\ge1/k)=0$ for each $k\in\mathbb N$, and $\{\cE_*>0\}=\bigcup_{k=1}^{\infty}\{\cE_*\ge1/k\}$ has probability zero.
Thus $\cE_*=0$ and $\cE_N\to0$ almost surely.

For the expectation bound, write $m_N:=\E[\cE_N]$ and distinguish sampling histories on which the error energy is at least $B/(N+1)$ from those below this moving threshold by setting $H_N:=\{\cE_N\ge B/(N+1)\}$.
For all sufficiently large $N$, $\cE_N\ge B/(N+1)\ge C_0D_N\tau_N^2$ on $H_N$ by \eqref{eq:abstract-rate-assumptions}, and the conditional expected decrease is bounded below by Proposition~\ref{prop:tail-or-decrease} as
\begin{equation}
\begin{array}{ll}
\mathbf1_{H_N}
\E[\cE_N-\cE_{N+1}\mid\F_N]
\ge
c_*\frac{\cE_N^2}{D_N}\mathbf1_{H_N}.
\end{array}
\label{eq:conditional-decrease-on-threshold-event}
\end{equation}
By Proposition~\ref{prop:energy-identity}, $\cE_N-\cE_{N+1}\ge0$, so multiplication of $\E[\cE_N-\cE_{N+1}\mid\F_N]$ by $\mathbf1_{H_N}$ cannot increase its value.
With $D_N\le C_{\gamma}(N+1)$ from \eqref{eq:abstract-rate-assumptions}, the unconditional decrease $m_N-m_{N+1}$ therefore satisfies, by \eqref{eq:conditional-decrease-on-threshold-event},
\begin{equation}
\begin{array}{ll}
m_N-m_{N+1}
=\E\!\bigl[\E[\cE_N-\cE_{N+1}\mid\F_N]\bigr]
\ge c_*\E\Bigl[\frac{\cE_N^2}{D_N}\mathbf1_{H_N}\Bigr]
\ge\frac{c_*}{C_{\gamma}(N+1)}\E[\cE_N^2\mathbf1_{H_N}].
\end{array}
\label{eq:mean-energy-decrease}
\end{equation}
Since $\cE_N<B/(N+1)$ on $H_N^c$, $\E[\cE_N\mathbf1_{H_N^c}]\le B/(N+1)$.
The remaining contribution $\E[\cE_N\mathbf1_{H_N}]$ is nonnegative, so, with $(x)_+:=\max\{x,0\}$,
\begin{equation}
\begin{array}{ll}
\E[\cE_N\mathbf1_{H_N}]
=
m_N-\E[\cE_N\mathbf1_{H_N^c}]
\ge
(m_N-\frac{B}{N+1})_+.
\end{array}
\label{eq:mean-energy-above-threshold}
\end{equation}
For $\cE_N\mathbf1_{H_N}$ and the constant function $1$, the Cauchy--Schwarz inequality reads $\E[\cE_N^2\mathbf1_{H_N}]\ge(\E[\cE_N\mathbf1_{H_N}])^2$.
The second moment $\E[\cE_N^2\mathbf1_{H_N}]$ is therefore at least $(m_N-B/(N+1))_+^2$ by \eqref{eq:mean-energy-above-threshold}, so the decrease estimate \eqref{eq:mean-energy-decrease} becomes, with $c:=c_*/C_{\gamma}$,
\[
\begin{array}{ll}
m_N-m_{N+1}
\ge
\frac{c}{N+1}
(m_N-\frac{B}{N+1})_+^2,
\end{array}
\]
or equivalently
\begin{equation*}
\begin{array}{ll}
m_{N+1}
\le
m_N-\frac{c}{N+1}
(m_N-\frac{B}{N+1})_+^2.
\end{array}
\end{equation*}
Since $0\le m_N\le F^2/\alpha$ and \eqref{eq:scalar-log-recursion} holds with $b=B$, Lemma~\ref{lem:scalar-log-bound} in Appendix~\ref{app:log} applies, and $m_N\le C/\log(N+2)$ for all sufficiently large $N$.
\end{proof}

The inverse-logarithmic estimate in Theorem~\ref{thm:abstract-main} is obtained from the conditional decrease bound \eqref{eq:conditional-correlation} in Proposition~\ref{prop:energy-identity}, where the iteration-dependent denominator in \eqref{eq:conditional-exact} is replaced by the uniform upper bound $MF^2/\alpha^2$.
Because this replacement discards information about the current error magnitude, the theorem provides a general upper bound for the expected error energy $\E[\cE_N]$, while the decay observed in a particular problem may be faster.

\subsection{Finite-dimensional solution families}

Suppose that the solutions $u(\xi)$ of the variational problem \eqref{eq:pvp} belong to a fixed finite-dimensional subspace of $\V$ for $\mu$-almost every $\xi$.
Since $u_0=0$ and $g_N=u(\Xi_N)-u_{N-1}(\Xi_N)$ by \eqref{eq:direction-as-error} in Lemma~\ref{lem:basic-error-relations}, the directions and their linear combinations $u_N(\xi)$ remain in this subspace by induction for $\Pp$-almost every sampling history.
The errors $R^{(N)}=u-u_N$ in \eqref{eq:error-definition} therefore also take values in this subspace for $\mu$-almost every parameter, so Lemma~\ref{lem:finite-correlation} applies with a fixed dimension bound and no approximation tail.
This setting permits an explicit $O(N^{-1})$ bound for $\E[\cE_N]$, the expectation over the sampling history of the parameter-averaged error energy $\cE_N$ in \eqref{eq:error-energy-definitions}, as stated in the following theorem.

\begin{theorem}
\label{thm:finite-dimensional}
Let $\{u_N\}_{N\ge0}$ be generated by Algorithm~\ref{alg:VS}, and let $R^{(N)}$ and $\cE_N$ be defined by \eqref{eq:error-definition} and \eqref{eq:error-energy-definitions}.
Suppose that there is a fixed subspace $W\subset\V$ such that
\[
\dim W=r,\quad 1\le r<\infty
\quad\mbox{and}\quad u(\xi)\in W
\]
for $\mu$-almost every $\xi$.
Then
\[
\cE_N\to0
\qquad \Pp\text{-a.s.}
\]
If $\cE_0>0$, then
\begin{equation}
\E[\cE_N]
\le
\frac{1}{\cE_0^{-1}+\tfrac{\alpha^2}{MrF^2}N}.
\label{eq:finite-rate}
\end{equation}
\end{theorem}

\begin{proof}
To apply Lemma~\ref{lem:finite-correlation} with the fixed space $W$, we first identify sampling histories along which every iterate remains in that space.
Choose a measurable set $\Omega_0\subset\Omega$ such that $\mu(\Omega_0)=1$ and $u(\xi)\in W$ for every $\xi\in\Omega_0$.
Since each $\Xi_N$ has distribution $\mu$ and there are countably many samples, $\Pp(\Xi_N\in\Omega_0,\ \forall N\ge1)=1$.
On this event, $u_0=0$ belongs to $W$, and each step preserves membership in $W$.
Indeed, if $g_1,\ldots,g_{N-1}\in W$, then their linear combination $u_{N-1}(\xi)$ in \eqref{eq:approximate-solution} belongs to $W$ for every $\xi\in\Omega$.
Since $\Xi_N\in\Omega_0$, the direction $g_N=u(\Xi_N)-u_{N-1}(\Xi_N)$ also belongs to $W$ by \eqref{eq:direction-as-error} in Lemma~\ref{lem:basic-error-relations}.
Thus $g_N\in W$, $G_N\subset W$, and $u_N(\xi)\in W$ for every $N\ge1$ and every $\xi\in\Omega$ by induction.
Consequently, $R^{(N)}(\xi)\in W$ for every $N$ and $\mu$-almost every $\xi$, so we may redefine the error to be zero on $\Omega\setminus\Omega_0$ to make it $W$-valued without changing $J_N$ or $\cE_N$.
The correlation satisfies $J_N\ge\cE_N^2/r$ by Lemma~\ref{lem:finite-correlation}, so the conditional decrease estimate \eqref{eq:conditional-correlation} in Proposition~\ref{prop:energy-identity} takes the form
\begin{equation}
\E[\cE_{N+1}\mid\F_N]
\le
\cE_N-\frac{\alpha^2}{MrF^2}\cE_N^2.
\label{eq:finite-dimensional-conditional}
\end{equation}

To convert this conditional inequality into a recursion for the expectation of the parameter-averaged error energy, let $m_N:=\E[\cE_N]$ and $c_r:=\alpha^2/(MrF^2)$.
The unconditional expectation of \eqref{eq:finite-dimensional-conditional} satisfies
\begin{equation}
m_{N+1}
\le
m_N-c_r\E[\cE_N^2]
\le
m_N-c_rm_N^2.
\label{eq:finite-dimensional-recursion}
\end{equation}
Here $\E[\cE_N^2]\ge m_N^2$ by Jensen's inequality, and $m_0=\cE_0$ because $\cE_0$ is deterministic.
If $m_0=0$, then $0\le\cE_N\le\cE_0=0$ along every sampling history by Proposition~\ref{prop:energy-identity}, so $\cE_N=0$ for every $N$.
For $m_0>0$, it suffices to prove \eqref{eq:finite-rate} up to the first index at which $m_N$ vanishes, if such an index exists, because all subsequent terms are then zero.
At every step with $m_{N+1}>0$, the recursion \eqref{eq:finite-dimensional-recursion} requires $0<m_{N+1}\le m_N(1-c_rm_N)$ and hence $0<c_rm_N<1$.
The reciprocals are therefore well defined, and $(1-t)^{-1}\ge1+t$ for $0\le t<1$ shows that
\[
\begin{array}{ll}
\frac1{m_{N+1}}
\ge
\frac1{m_N(1-c_rm_N)}
\ge
\frac1{m_N}+c_r.
\end{array}
\]
Thus $1/m_N$ increases by at least $c_r$ at each step for which $m_N>0$, so $1/m_N\ge1/m_0+c_rN$ and \eqref{eq:finite-rate} holds.
The estimate \eqref{eq:finite-rate} also holds when $m_N=0$.

By Proposition~\ref{prop:energy-identity}, $0\le\cE_{N+1}\le\cE_N$, so the sequence has an almost-sure limit $\cE_*\ge0$.
Because $0\le\cE_N\le\cE_0$, dominated convergence and the expectation bound \eqref{eq:finite-rate} ensure $\E[\cE_*]=\lim_{N\to\infty}\E[\cE_N]=0$.
The nonnegative random variable $\cE_*$ is therefore zero almost surely.
\end{proof}

\section{Problems with separable parameter dependence}
\label{sec:finite-term}

To verify the finite-dimensional error approximation conditions of Theorem~\ref{thm:abstract-main} for problems with finite-term separated bilinear forms and right-hand sides, we first construct deterministic spaces that approximate the solutions $u(\xi)$ of \eqref{eq:pvp}, with explicit dimension and error bounds.
Since $u_N(\xi)$ belongs to the generated direction space $G_N$ in \eqref{eq:generated-direction-space}, these spaces enlarged by $G_N$ approximate the errors $R^{(N)}=u-u_N$ in \eqref{eq:error-definition}.
Theorem~\ref{thm:abstract-main} then ensures almost-sure convergence to zero of the parameter-averaged error energy $\cE_N$ in \eqref{eq:error-energy-definitions} and bounds $\E[\cE_N]$.

\subsection{Approximation of the solution family}

Throughout this section, assume that the bilinear form $a(\cdot,\cdot;\xi)$ in the variational problem \eqref{eq:pvp} admits, for every $\xi\in\Omega$ and all $v,w\in\V$, the representation
\begin{equation}
a(v,w;\xi)=a_0(v,w)+\compactsum_{q=1}^{Q}\theta_q(\xi)a_q(v,w)
\label{eq:separable-a}
\end{equation}
and that the right-hand side $f(\xi)\in\Vp$ in \eqref{eq:pvp} admits the representation
\begin{equation}
f(\xi)=\compactsum_{\ell=0}^{L}\phi_\ell(\xi)F_\ell.
\label{eq:separable-f}
\end{equation}
Here $Q\ge1$ and $L\ge0$ are integers, $\theta_q:\Omega\to\RR$ for $q=1,\ldots,Q$ are scalar coefficient functions, $F_\ell\in\Vp$ for $\ell=0,\ldots,L$, and each $\phi_\ell:\Omega\to\RR$ is bounded and Borel measurable.
The forms $a_0,\ldots,a_Q$ are parameter-independent bilinear forms on $\V\times\V$, with $a_q$ bounded for $q=1,\ldots,Q$.
To express the variational problem in operator form, for each $\xi\in\Omega$ define $A(\xi):\V\to\Vp$ and, for $q=1,\ldots,Q$, define $A_q:\V\to\Vp$ by
\begin{equation}
\langle A(\xi)v,w\rangle=a(v,w;\xi)
\quad\mbox{and}\quad
\langle A_qv,w\rangle=a_q(v,w).
\label{eq:operator-definitions}
\end{equation}

The operators $A(\xi)$ in \eqref{eq:operator-definitions} lie in a finite-dimensional affine space by \eqref{eq:separable-a} and satisfy $\|A(\xi)\|_{\mathcal L(\V,\Vp)}\le M$ for every $\xi\in\Omega$ by \eqref{eq:bounded}.
This bounded family has a finite cover by operator-norm balls of radius $\alpha/2$ centered at operators from the same family.
Since the covering radius is $\alpha/2$ and $\|A(\xi)^{-1}\|_{\mathcal L(\Vp,\V)}\le1/\alpha$ by \eqref{eq:coercive}, the local Neumann expansions converge about these reference operators.
Their truncations, applied to $f(\xi)$ in \eqref{eq:separable-f}, generate the approximation spaces below.

\begin{proposition}
\label{prop:solution-approx}
Let $u(\xi)$ be the solution of \eqref{eq:pvp}, and let $F$ be defined by \eqref{eq:fbound}.
There is an integer $K\ge1$, independent of the truncation degree and the sampling history, such that, for every integer $p\ge0$, there is a finite-dimensional space $W_p\subset\V$, independent of the sampling history, satisfying
\begin{equation}
\begin{cases}
\dim W_p\le K(L+1)\binom{Q+p}{p},
\\[2mm]
\sup_{\xi\in\Omega}\operatorname{dist}_{\V}(u(\xi),W_p)
\le\frac{F}{\alpha}2^{-p}.
\end{cases}
\label{eq:solution-space-approximation}
\end{equation}
\end{proposition}

\begin{proof}
First choose the reference operators needed for the local expansions.
Fix $\xi_0\in\Omega$ and set $\mathcal X:=\operatorname{span}\{A_1,\ldots,A_Q\}\subset\mathcal L(\V,\Vp)$.
The space $\mathcal X$ is finite-dimensional because $Q$ is finite, and the representation of each shifted operator in $\mathcal X$ reads, by \eqref{eq:separable-a},
\[
\begin{array}{ll}
A(\xi)-A(\xi_0)
=
\compactsum_{q=1}^{Q}
\bigl(\theta_q(\xi)-\theta_q(\xi_0)\bigr)A_q
\in\mathcal X
\qquad \forall\,\xi\in\Omega.
\end{array}
\]
The bound $\|A(\xi)\|_{\mathcal L(\V,\Vp)}\le M$ from \eqref{eq:bounded} holds for every $\xi\in\Omega$, so $\|A(\xi)-A(\xi_0)\|_{\mathcal L(\V,\Vp)}\le2M$.
Hence the shifted family $\mathcal S:=\{A(\xi)-A(\xi_0):\xi\in\Omega\}$ is bounded in $\mathcal X$ and has a finite cover by balls of radius $\alpha/4$.
To replace centers that need not lie in $\mathcal S$, select an element $A(\zeta^{(j)})-A(\xi_0)$ of $\mathcal S$ from each covering ball that intersects it, with $\zeta^{(1)},\ldots,\zeta^{(K)}\in\Omega$.
For any $\xi\in\Omega$, its shifted operator and the selected representative of a ball containing it are both within $\alpha/4$ of that ball's center.
Thus
\begin{equation}
\begin{array}{ll}
\|A(\xi)-A(\zeta^{(j)})\|_{\mathcal L(\V,\Vp)}
&=
\|
[A(\xi)-A(\xi_0)]
-
[A(\zeta^{(j)})-A(\xi_0)]
\|_{\mathcal L(\V,\Vp)}
\\[2mm]
&\le
\frac{\alpha}{4}
+
\frac{\alpha}{4}
=
\frac{\alpha}{2}.
\end{array}
\label{eq:operator-cover}
\end{equation}
The reference parameters and their number $K$ are fixed independently of $p$ and the sampling history.

For fixed $\zeta\in\Omega$, the equation $A(\zeta)v=h$ has a unique solution $v\in\V$ for every $h\in\Vp$ by the Lax--Milgram theorem for the bounded coercive form $a(\cdot,\cdot;\zeta)$, so $A(\zeta):\V\to\Vp$ is bijective.
Moreover, $\alpha\|v\|_{\V}\le\|A(\zeta)v\|_{\Vp}$ for every $v\in\V$ by \eqref{eq:coercive}, and hence $A(\zeta)$ is an isomorphism with $\|A(\zeta)^{-1}\|_{\mathcal L(\Vp,\V)}\le1/\alpha$.
Define the perturbation operator $T_\zeta(\xi):=A(\zeta)^{-1}\bigl(A(\xi)-A(\zeta)\bigr)$.
For any $\xi\in\Omega$ satisfying
\begin{equation}
\|A(\xi)-A(\zeta)\|_{\mathcal L(\V,\Vp)}
\le\frac{\alpha}{2},
\label{eq:local-operator-closeness}
\end{equation}
the inverse bound gives
\[
\begin{array}{ll}
\|T_\zeta(\xi)\|_{\mathcal L(\V)}
&\le
\|A(\zeta)^{-1}\|_{\mathcal L(\Vp,\V)}
\|A(\xi)-A(\zeta)\|_{\mathcal L(\V,\Vp)}
\le
\frac1\alpha\frac{\alpha}{2}
=
\frac12.
\end{array}
\]
Since $A(\xi)-A(\zeta)=A(\zeta)T_\zeta(\xi)$, the operator factors as $A(\xi)=A(\zeta)(I+T_\zeta(\xi))$, where $\|T_\zeta(\xi)\|_{\mathcal L(\V)}\le1/2$ ensures convergence of the Neumann series for $(I+T_\zeta(\xi))^{-1}$ in operator norm.
For every integer $p\ge0$, the first $p+1$ terms and tail of this series are
\[
\begin{array}{ll}
(I+T_\zeta(\xi))^{-1}
=
\compactsum_{k=0}^{p}(-T_\zeta(\xi))^k
+
\compactsum_{k=p+1}^{\infty}(-T_\zeta(\xi))^k.
\end{array}
\]
The tail $\mathcal R_{\zeta,p}(\xi):=\compactsum_{k=p+1}^{\infty}(-T_\zeta(\xi))^k$ satisfies the geometric bound
\[
\begin{array}{ll}
\|\mathcal R_{\zeta,p}(\xi)\|_{\mathcal L(\V)}
\le
\compactsum_{k=p+1}^{\infty}2^{-k}
=
2^{-p}.
\end{array}
\]
After multiplication of the expansion of $(I+T_\zeta(\xi))^{-1}$ on the right by $A(\zeta)^{-1}$, the local inverse has the representation
\begin{equation}
A(\xi)^{-1}
=
\bigl[
\compactsum_{k=0}^{p}(-T_\zeta(\xi))^k
+\mathcal R_{\zeta,p}(\xi)
\bigr]A(\zeta)^{-1}
\quad\mbox{and}\quad
\|\mathcal R_{\zeta,p}(\xi)\|_{\mathcal L(\V)}
\le2^{-p}.
\label{eq:local-inverse-expansion}
\end{equation}

To approximate the solution family by truncating the local expansions of $A(\xi)^{-1}$ in \eqref{eq:local-inverse-expansion}, fix an integer $p\ge0$.
The covering estimate \eqref{eq:operator-cover} ensures that, for every $\xi\in\Omega$, the reference parameters $\zeta^{(1)},\ldots,\zeta^{(K)}$ chosen above include at least one $\zeta^{(j)}$ such that
\begin{equation}
\|A(\xi)-A(\zeta^{(j)})\|_{\mathcal L(\V,\Vp)}
\le\frac{\alpha}{2}.
\label{eq:operator-cover-selection}
\end{equation}
Choose such a reference parameter $\zeta=\zeta^{(j)}$ and apply the expansion \eqref{eq:local-inverse-expansion} to $f(\xi)$ to obtain
\begin{equation}
u(\xi)
=
\compactsum_{k=0}^{p}
(-T_\zeta(\xi))^k
A(\zeta)^{-1}f(\xi)
+
e_{\zeta,p}(\xi),
\label{eq:local-neumann}
\end{equation}
where $e_{\zeta,p}(\xi)=\mathcal R_{\zeta,p}(\xi)A(\zeta)^{-1}f(\xi)\in\V$ is the vector remainder induced by the operator remainder in \eqref{eq:local-inverse-expansion}.
Since $\|A(\zeta)^{-1}\|_{\mathcal L(\Vp,\V)}\le1/\alpha$ and $\|f(\xi)\|_{\Vp}\le F$ by \eqref{eq:fbound}, the vector $A(\zeta)^{-1}f(\xi)$ has norm at most $F/\alpha$.
The solution error $\mathcal R_{\zeta,p}(\xi)A(\zeta)^{-1}f(\xi)$ is therefore bounded using $\|\mathcal R_{\zeta,p}(\xi)\|_{\mathcal L(\V)}\le2^{-p}$ from \eqref{eq:local-inverse-expansion} as
\begin{equation}
\|e_{\zeta,p}(\xi)\|_{\V}
\le\|\mathcal R_{\zeta,p}(\xi)\|_{\mathcal L(\V)}
\|A(\zeta)^{-1}f(\xi)\|_{\V}
\le\frac{F}{\alpha}2^{-p}.
\label{eq:local-tail-bound}
\end{equation}
For fixed $\zeta$, the truncated sums in \eqref{eq:local-neumann} lie in a finite-dimensional space determined by their spatial components.
Define the bounded operators $B_q^{(\zeta)}:=A(\zeta)^{-1}A_q$ on $\V$ and the vectors $b_\ell^{(\zeta)}:=A(\zeta)^{-1}F_\ell$ in $\V$.
The separated representations \eqref{eq:separable-a} and \eqref{eq:separable-f} express the perturbation and the transformed right-hand side as
\[
\begin{cases}
T_\zeta(\xi)
=
\compactsum_{q=1}^{Q}
\bigl(\theta_q(\xi)-\theta_q(\zeta)\bigr)
B_q^{(\zeta)},\\[2mm]
A(\zeta)^{-1}f(\xi)
=
\compactsum_{\ell=0}^{L}
\phi_\ell(\xi)b_\ell^{(\zeta)}.
\end{cases}
\]
To collect the operator products in the truncated sum, let $\nu=(\nu_1,\ldots,\nu_Q)\in\{0,1,\ldots\}^Q$ record their factor multiplicities and set $|\nu|:=\nu_1+\cdots+\nu_Q$.
For $\ell=0,\ldots,L$ and $|\nu|\ge1$, let $v_{\zeta,\nu,\ell}$ be the sum of $B_{i_1}^{(\zeta)}\cdots B_{i_{|\nu|}}^{(\zeta)}b_\ell^{(\zeta)}$ over all index sequences $(i_1,\ldots,i_{|\nu|})$ in which each $q\in\{1,\ldots,Q\}$ occurs exactly $\nu_q$ times, and set $v_{\zeta,0,\ell}:=b_\ell^{(\zeta)}$ for the term with no operator factors.
The expansion of the truncated sum in \eqref{eq:local-neumann}, grouped by the multiplicities $\nu$ of the scalar factors, belongs to $W_{\zeta,p}:=\operatorname{span}\{v_{\zeta,\nu,\ell}:|\nu|\le p,\ 0\le\ell\le L\}$, as detailed in Appendix~\ref{app:neumann}.
Every operator product retains its order, so no commutativity of the $B_q^{(\zeta)}$ is required.
The vectors $v_{\zeta,\nu,\ell}$ are independent of $\xi$ and indexed by $\compactsum_{k=0}^{p}\binom{Q+k-1}{k}=\binom{Q+p}{p}$ possible $Q$-dimensional multi-indices with $|\nu|\le p$ and $L+1$ choices of $\ell$, so
\begin{equation*}
\begin{array}{ll}
\dim W_{\zeta,p}
\le
(L+1)\binom{Q+p}{p}.
\end{array}
\end{equation*}

The space $W_p:=\compactsum_{j=1}^{K}W_{\zeta^{(j)},p}$ combines the local spaces to approximate $u(\xi)$ for every $\xi\in\Omega$.
Since $W_p$ depends only on the fixed reference parameters and the problem data, it is independent of the sampling history, and its dimension is bounded by the sum of the local dimensions, so
\[
\begin{array}{ll}
\dim W_p
\le
\compactsum_{j=1}^{K}\dim W_{\zeta^{(j)},p}
\le
K(L+1)\binom{Q+p}{p},
\end{array}
\]
which proves the dimension bound in \eqref{eq:solution-space-approximation}.
For every $\xi\in\Omega$, a reference parameter $\zeta^{(j)}$ satisfying \eqref{eq:operator-cover-selection} supplies a truncated sum in \eqref{eq:local-neumann} that belongs to $W_{\zeta^{(j)},p}\subset W_p$.
Its remainder satisfies \eqref{eq:local-tail-bound}, so $\operatorname{dist}_{\V}(u(\xi),W_p)\le(F/\alpha)2^{-p}$ uniformly in $\xi$.
The supremum over $\Omega$ proves the approximation estimate in \eqref{eq:solution-space-approximation}.
\end{proof}

\subsection{Convergence of the VS iteration}

The deterministic spaces $W_p$ from Proposition~\ref{prop:solution-approx} approximate the solutions $u(\xi)$ of \eqref{eq:pvp}, while Theorem~\ref{thm:abstract-main} concerns the errors $R^{(N)}$ in \eqref{eq:error-definition}.
Since $u_N(\xi)$ belongs to $G_N$ in \eqref{eq:generated-direction-space}, every $w\in W_p$ determines an approximation $w-u_N(\xi)\in W_p+G_N$ with $R^{(N)}(\xi)-(w-u_N(\xi))=u(\xi)-w$.
Thus $\operatorname{dist}_{\V}(R^{(N)}(\xi),W_p+G_N)\le\operatorname{dist}_{\V}(u(\xi),W_p)$, and a logarithmically growing degree $p$ meets both the qualitative and quantitative conditions of Theorem~\ref{thm:abstract-main} for $\cE_N\to0$ almost surely and a bound on $\E[\cE_N]$.

\begin{theorem}
\label{thm:finite-term-main}
Let $\{u_N\}_{N\ge0}$ be generated by Algorithm~\ref{alg:VS}, and let $R^{(N)}$ and $\cE_N$ be defined by \eqref{eq:error-definition} and \eqref{eq:error-energy-definitions}.
Then
\begin{equation*}
\begin{array}{ll}
\cE_N\to0
\qquad \Pp\text{-a.s.}
\end{array}
\end{equation*}
Moreover, there is a constant $C>0$, independent of $N$, such that
\begin{equation}
\E[\cE_N]
\le
\tfrac{C}{\log(N+2)}
\label{eq:finite-term-rate}
\end{equation}
for all sufficiently large $N$.
\end{theorem}

\begin{proof}
We verify the projection measurability condition \eqref{eq:projection-measurability} and the approximation bounds \eqref{eq:abstract-approx} in Theorem~\ref{thm:abstract-main} by adjoining the generated direction space $G_N$ in \eqref{eq:generated-direction-space} to the deterministic approximation spaces $W_p$ from Proposition~\ref{prop:solution-approx}.
For every $N\ge0$, set
\begin{equation}
p_N:=\lceil\log_2(N+2)\rceil,
\label{eq:pN-choice}
\end{equation}
let $W_{p_N}$ be the space supplied by Proposition~\ref{prop:solution-approx}, and define
\[
\begin{array}{ll}
Z_N:=W_{p_N}+G_N,
\quad
D_N:=N+K(L+1)\binom{Q+p_N}{p_N}
\quad\mbox{and}\quad 
\tau_N:=\frac{F}{\alpha}2^{-p_N}.
\end{array}
\]
Here $K$ is the fixed integer in Proposition~\ref{prop:solution-approx}, and $G_N$ is the generated direction space in \eqref{eq:generated-direction-space}.
Thus $D_N$ is a deterministic positive integer and $\tau_N$ is deterministic and nonnegative.
A fixed basis of $W_{p_N}$ followed by $g_1,\ldots,g_N$ forms an $\F_N$-measurable spanning list for $Z_N$, since the directions are measurable by Proposition~\ref{prop:iteration-measurability}.
For a nonempty list, Lemma~\ref{lem:measurable-GS} constructs an orthogonal projection $P_N$ onto $Z_N$ satisfying \eqref{eq:projection-measurability} by Gram--Schmidt, as detailed in Appendix~\ref{app:measurability}, while an empty list corresponds to $Z_N=\{0\}$ and the zero projection.
Since $\dim G_N\le N$, the dimension of $Z_N$ satisfies, by Proposition~\ref{prop:solution-approx},
\[
\begin{array}{ll}
\dim Z_N
\le
N+K(L+1)\binom{Q+p_N}{p_N}
=D_N.
\end{array}
\]
For each $\xi\in\Omega$, choose $w_N(\xi)\in W_{p_N}$ with $\|u(\xi)-w_N(\xi)\|_{\V}\le\tau_N$ by \eqref{eq:solution-space-approximation} in Proposition~\ref{prop:solution-approx}.
Because $u_N(\xi)\in G_N$, the vector $w_N(\xi)-u_N(\xi)$ belongs to $Z_N=W_{p_N}+G_N$ and bounds the approximation error of $R^{(N)}(\xi)$ through
\[
\begin{array}{ll}
\operatorname{dist}_{\V}(R^{(N)}(\xi),Z_N)
&\le\|R^{(N)}(\xi)-[w_N(\xi)-u_N(\xi)]\|_{\V}\\[2mm]
&=\|u(\xi)-w_N(\xi)\|_{\V}\le\tau_N.
\end{array}
\]
This estimate holds for every $\xi\in\Omega$, so $\sup_{\xi\in\Omega}\operatorname{dist}_{\V}(R^{(N)}(\xi),Z_N)\le\tau_N$.
The dimension and error bounds verify \eqref{eq:abstract-approx}.

It remains to check that the dimension bound $D_N$ and tolerance $\tau_N$ satisfy \eqref{eq:abstract-convergence-assumptions} for the degree $p_N$ in \eqref{eq:pN-choice}.
For fixed $Q$ and $p_N=O(\log N)$, the binomial coefficient in $D_N$ satisfies
\[
\begin{array}{ll}
\binom{Q+p_N}{p_N}
=
\binom{Q+p_N}{Q}
=O(p_N^Q)
=O((\log N)^Q),
\end{array}
\]
so $D_N=N+O((\log N)^Q)$ and hence $D_N$ is eventually bounded by a constant multiple of $N+1$.
At the same time, the ceiling in \eqref{eq:pN-choice} ensures that $2^{-p_N}\le2^{-\log_2(N+2)}=1/(N+2)$.
Consequently, there are constants $C_{\gamma},C_\tau>0$ such that, for all sufficiently large $N$,
\begin{equation}
D_N\le C_{\gamma}(N+1)
\quad\mbox{and}\quad
\tau_N\le\frac{C_\tau}{N+2}.
\label{eq:DN-tauN-bounds}
\end{equation}
In particular, the product $D_N\tau_N^2$ tends to zero by \eqref{eq:DN-tauN-bounds}, since
\[
\begin{array}{ll}
D_N\tau_N^2
\le
C_{\gamma}C_\tau^2\frac{N+1}{(N+2)^2}
\to0.
\end{array}
\]
The series in \eqref{eq:abstract-convergence-assumptions} diverges by comparison with the harmonic series after choosing $N_0$ so that $D_N\le C_{\gamma}(N+1)$ for every $N\ge N_0$, since
\[
\begin{array}{ll}
\compactsum_{N=0}^{\infty}\frac1{D_N}
\ge
\compactsum_{N=N_0}^{\infty}\frac1{D_N}
\ge
\frac1{C_{\gamma}}
\compactsum_{N=N_0}^{\infty}\frac1{N+1}
=
\infty.
\end{array}
\]
The almost-sure convergence assertion now follows from Theorem~\ref{thm:abstract-main}.

For the bound on $\E[\cE_N]$, it remains to verify the quantitative conditions \eqref{eq:abstract-rate-assumptions} in Theorem~\ref{thm:abstract-main}.
For the threshold constant $C_0$ from Proposition~\ref{prop:tail-or-decrease}, the product $C_0D_N\tau_N^2$ and dimension $D_N$ satisfy, for all sufficiently large $N$ and by \eqref{eq:DN-tauN-bounds},
\[
\begin{cases}
C_0D_N\tau_N^2
\le C_0C_{\gamma}C_\tau^2\frac{N+1}{(N+2)^2}
\le\frac{C_0C_{\gamma}C_\tau^2}{N+1},\\[2mm]
D_N\le C_{\gamma}(N+1).
\end{cases}
\]
Thus \eqref{eq:abstract-rate-assumptions} holds with $B:=C_0C_{\gamma}C_\tau^2$, so the expected-error assertion of Theorem~\ref{thm:abstract-main} proves \eqref{eq:finite-term-rate}.
\end{proof}

For example, the data assumptions of Theorem~\ref{thm:finite-term-main} hold for the representations \eqref{eq:separable-a}--\eqref{eq:separable-f} when $a_0,\ldots,a_Q$ are bounded bilinear forms, the coefficients $\theta_q$ are Lipschitz continuous on the compact parameter space $\Omega$, the coefficients $\phi_\ell$ are bounded and Borel measurable, and the resulting bilinear form is symmetric and uniformly coercive.
Since each $\theta_q$ is continuous on the compact set $\Omega$, it is bounded and Borel measurable.
In \eqref{eq:separable-a}--\eqref{eq:separable-f}, these functions and the bounded Borel functions $\phi_\ell$ multiply fixed bounded bilinear forms or fixed functionals.
The resulting finite sums are therefore Borel measurable as maps $\xi\mapsto a(v,w;\xi)$ for fixed $v,w\in\V$ and $\xi\mapsto f(\xi)$ into $\Vp$, and satisfy the uniform bounds \eqref{eq:bounded} and \eqref{eq:fbound} for suitable constants.
These conditions are sufficient to apply the theorem, whose proof uses the finite-dimensional structure and uniform boundedness of the operator family without requiring additional continuity of the individual $\theta_q$.

\section{Parameter-independent operators}
\label{sec:fixed-operator}

Suppose throughout this section that the bilinear form $a(\cdot,\cdot;\xi)$ in the variational problem \eqref{eq:pvp} is independent of the parameter, so that
\begin{equation}
a(v,w;\xi)=a(v,w)
\qquad
\forall\,\xi\in\Omega,
\quad \forall\,v,w\in\V.
\label{eq:fixed-operator}
\end{equation}
All parameters then share the energy inner product $(v,w)_a:=a(v,w)$ and its induced norm $\|v\|_a:=a(v,v)^{1/2}$.

Proposition~\ref{prop:fixed-orthogonal} proves that $a(u(\xi)-u_N(\xi),v)=0$ for every $v$ in the generated direction space $G_N$ in \eqref{eq:generated-direction-space}.
Since $u_N(\xi)\in G_N$, the VS approximation $u_N(\xi)$ in \eqref{eq:approximate-solution} is the $a$-orthogonal projection of the solution $u(\xi)$ of \eqref{eq:pvp} onto $G_N$ and minimizes $\|u(\xi)-v\|_a$ over $v\in G_N$.
If $u(\xi)$ lies in a fixed finite-dimensional subspace for $\mu$-almost every $\xi$, the same proposition shows that $G_N$ reaches the smallest such subspace after an almost surely finite number of samples, so that $u_N(\xi)=u(\xi)$ for $\mu$-almost every $\xi$ at every subsequent iteration.
For a separated right-hand side $f(\xi)$, Proposition~\ref{prop:fixed-separable} identifies finite-dimensional spaces containing the solutions $u(\xi)$ and the coefficient functions $\eta_N$.

\subsection{Orthogonal projection and finite-step recovery}

By Lemma~\ref{lem:basic-error-relations}, the error $R^{(N)}(\xi)$ is orthogonal to the new direction $g_N$, as stated in \eqref{eq:one-step-orthogonality}.
To identify $u_N(\xi)$ as the $a$-orthogonal projection of $u(\xi)$ onto the generated direction space $G_N$ in \eqref{eq:generated-direction-space}, i.e., $G_N=\operatorname{span}\{g_1,\ldots,g_N\}$, the error must also be orthogonal to the earlier directions.
If $a(R^{(N-1)}(\xi),g_j)=0$ for all $\xi\in\Omega$ and $j<N$, evaluation at $\xi=\Xi_N$ establishes $a(g_N,g_j)=0$ because $a$ is parameter-independent and $g_N=R^{(N-1)}(\Xi_N)$.
The correction $\eta_N(\xi)g_N$ in \eqref{eq:error-update} therefore preserves $a(R^{(N)}(\xi),g_j)=0$ for $j<N$.
The proof of Proposition~\ref{prop:fixed-orthogonal} also identifies $G_N=\operatorname{span}\{u(\Xi_1),\ldots,u(\Xi_N)\}$, which determines when $u_N(\xi)=u(\xi)$ for $\mu$-almost every $\xi$.

\begin{proposition}
\label{prop:fixed-orthogonal}
Let $\{u_N\}$, $\{g_N\}$, and $\{\eta_N\}$ be generated by Algorithm~\ref{alg:VS}, and let $R^{(N)}$ and $G_N$ be defined by \eqref{eq:error-definition} and \eqref{eq:generated-direction-space}.
Then all nonzero directions satisfy
\begin{equation}
a(g_i,g_j)=0,
\qquad i\ne j.
\label{eq:fixed-direction-orthogonality}
\end{equation}
For every $N$, every $\xi\in\Omega$, and every $j=1,\ldots,N$, we have
\begin{equation}
a(R^{(N)}(\xi),g_j)=0.
\label{eq:fixed-error-orthogonality}
\end{equation}
If $g_N\ne0$, then
\begin{equation}
\eta_N(\xi)
=
\tfrac{\langle f(\xi),g_N\rangle}{a(g_N,g_N)}.
\label{eq:fixed-eta-basic}
\end{equation}
For every $N$ and $\xi\in\Omega$, the approximation $u_N(\xi)$ and the error $R^{(N)}(\xi)$ satisfy
\begin{equation}
u_N(\xi)=P_{G_N}^{a}u(\xi)
\quad\mbox{and}\quad R^{(N)}(\xi)=(I-P_{G_N}^{a})u(\xi),
\label{eq:projection-error}
\end{equation}
where $P_{G_N}^{a}$ denotes the $a$-orthogonal projection onto $G_N$. 
The approximation $u_N(\xi)$ also has the best-approximation property
\begin{equation}
\|u(\xi)-u_N(\xi)\|_a
=
\inf_{v\in G_N}\|u(\xi)-v\|_a.
\label{eq:fixed-best-approximation}
\end{equation}

For the recovery assertion, let $W_*\subset\mathcal V$ be the essential linear span of the solution family, namely, the smallest closed subspace containing $u(\xi)$ for $\mu$-almost every $\xi\in\Omega$, and suppose in addition that $\dim W_*=r_*<\infty$.
Then, with probability one, there exists a finite random index $T$ such that $G_T=W_*$, and
\[
u_N(\xi)=u(\xi)
\qquad
\text{for $\mu$-almost every }\xi\in\Omega,
\quad N\ge T.
\]
On the same event of probability one, at most $r_*$ nonzero directions are generated, although the total number of samples required to reach $T$ may exceed $r_*$.
\end{proposition}

\begin{proof}
When $g_N=0$, Algorithm~\ref{alg:VS} sets $\eta_N\equiv0$, so $G_N=G_{N-1}$ and $u_N=u_{N-1}$, and such an update preserves all previously established orthogonality relations.
We therefore prove the direction and error orthogonality together by considering the nonzero updates.
At $N=1$, \eqref{eq:one-step-orthogonality} establishes \eqref{eq:fixed-error-orthogonality}, and there are no direction pairs.
Suppose that \eqref{eq:fixed-error-orthogonality} holds after step $N-1$.
The new direction $g_N=R^{(N-1)}(\Xi_N)$ inherits this orthogonality, since $a(g_N,g_j)=a(R^{(N-1)}(\Xi_N),g_j)=0$ for $j=1,\ldots,N-1$.
Consequently, the correction $\eta_N(\xi)g_N$ in \eqref{eq:error-update} preserves the earlier relations, since $a(R^{(N)}(\xi),g_j)=a(R^{(N-1)}(\xi),g_j)-\eta_N(\xi)a(g_N,g_j)=0$ for $j<N$.
By Lemma~\ref{lem:basic-error-relations}, the update satisfies $a(R^{(N)}(\xi),g_N)=0$ as stated in \eqref{eq:one-step-orthogonality}.
This completes the induction for \eqref{eq:fixed-direction-orthogonality}--\eqref{eq:fixed-error-orthogonality}.

For a nonzero direction $g_N$, the term $u_{N-1}(\xi)\in G_{N-1}$ contributes nothing to the coefficient numerator because $g_N$ is orthogonal to $G_{N-1}$, so
\[
\begin{array}{ll}
a(R^{(N-1)}(\xi),g_N)
&=a(u(\xi)-u_{N-1}(\xi),g_N)=a(u(\xi),g_N)=\langle f(\xi),g_N\rangle.
\end{array}
\]
Since $a(R^{(N-1)}(\xi),g_N)=\langle f(\xi),g_N\rangle$, the coefficient formula \eqref{eq:eta-residual} is precisely \eqref{eq:fixed-eta-basic}.
Moreover, $u_N(\xi)\in G_N$ and $u(\xi)-u_N(\xi)$ is orthogonal to $G_N$ by \eqref{eq:fixed-error-orthogonality}, which characterizes the projection in \eqref{eq:projection-error}. 

For any $v\in G_N$, the vector $u_N(\xi)-v$ belongs to $G_N$ and is therefore orthogonal to $R^{(N)}(\xi)$.
Hence
\[
\begin{array}{ll}
\|u(\xi)-v\|_a^2
&=\|R^{(N)}(\xi)+u_N(\xi)-v\|_a^2\\[2mm]
&=\|R^{(N)}(\xi)\|_a^2+\|u_N(\xi)-v\|_a^2
\ge\|R^{(N)}(\xi)\|_a^2.
\end{array}
\]
Equality is attained for $v=u_N(\xi)$, proving \eqref{eq:fixed-best-approximation}.

To prove that $G_N$ reaches $W_*$ after an almost surely finite number of samples, suppose that $\dim W_*=r_*<\infty$.
The direction $g_N=R^{(N-1)}(\Xi_N)=(I-P_{G_{N-1}}^a)u(\Xi_N)$ is the component of $u(\Xi_N)$ orthogonal to $G_{N-1}$ by \eqref{eq:projection-error}, so adjoining $g_N$ or $u(\Xi_N)$ to $G_{N-1}$ produces the same space, i.e.,
\begin{equation}
\operatorname{span}\{G_{N-1},g_N\}=\operatorname{span}\{G_{N-1},u(\Xi_N)\}.
\label{eq:fixed-span-recursion}
\end{equation}
Since each $\Xi_n$ has law $\mu$ and $u(\xi)\in W_*$ for $\mu$-almost every $\xi$, the events $\{u(\Xi_n)\in W_*\}$ each have probability one, as does their countable intersection over $n\ge1$.
On this event, the recursion \eqref{eq:fixed-span-recursion}, starting from $G_0=\{0\}$, identifies every generated space as
\[
G_N=\operatorname{span}\{u(\Xi_1),\ldots,u(\Xi_N)\}\subset W_*.
\]

If $r_*=0$, recovery holds with $T=0$, so it remains to consider $r_*>0$ and show that the sampled span reaches $W_*$ after an almost surely finite number of samples.
As long as $G_N$ is a proper subspace of $W_*$, the minimality of $W_*$ ensures that $\mu\{\xi\in\Omega:u(\xi)\notin G_N\}>0$.
A sample whose solution lies in $G_N$ produces a zero direction and leaves the space unchanged, so the probability of leaving $G_N$ stays positive and constant until the next nonzero direction.
For any integer $m\ge1$, the next $m$ independent samples all have their solutions in $G_N$ with conditional probability $\bigl(\mu\{\xi\in\Omega:u(\xi)\in G_N\}\bigr)^m$ given $\F_N$.
Since $\mu\{\xi\in\Omega:u(\xi)\in G_N\}<1$, this probability tends to zero as $m\to\infty$, so a solution outside $G_N$ is sampled after an almost surely finite number of trials.
Its residual direction is nonzero and increases the dimension of the generated space by one.
Since at most $r_*=\dim W_*$ increases are needed and each occurs after an almost surely finite waiting time, $G_T=W_*$ for some finite random index $T$ with probability one.
For all $N\ge T$, the projection $u_N(\xi)=P_{G_N}^a u(\xi)$ then equals $u(\xi)$ for $\mu$-almost every $\xi$.
\end{proof}

\subsection{Separable right-hand sides}

For a right-hand side of the form \eqref{eq:fixed-separable-f}, the variational problem \eqref{eq:pvp} with $f(\xi)$ replaced by $F_\ell$ has a parameter-independent solution because $a$ is fixed.
The combination of these fixed solutions with coefficients $\phi_\ell(\xi)$ solves \eqref{eq:pvp} with right-hand side $\compactsum_{\ell=0}^{L}\phi_\ell(\xi)F_\ell=f(\xi)$ by linearity in the first argument of $a$.
This combination equals the unique solution $u(\xi)$, which therefore lies in the span of the fixed solutions.
The following proposition establishes the representation and, through \eqref{eq:fixed-eta-basic} from Proposition~\ref{prop:fixed-orthogonal}, places every VS coefficient function $\eta_N$ in $\operatorname{span}\{\phi_0,\ldots,\phi_L\}$.

\begin{proposition}
\label{prop:fixed-separable}
Suppose that the right-hand side in \eqref{eq:pvp} has the representation
\begin{equation}
f(\xi)=\compactsum_{\ell=0}^{L}\phi_\ell(\xi)F_\ell,
\label{eq:fixed-separable-f}
\end{equation}
where $L\ge0$ is an integer and, for $\ell=0,\ldots,L$, $\phi_\ell:\Omega\to\RR$ is a scalar coefficient function and $F_\ell\in\V'$ is a fixed linear functional.
For $\ell=0,\ldots,L$, let $w_\ell\in\V$ be the unique solution of
\begin{equation}
a(w_\ell,v)=\langle F_\ell,v\rangle
\qquad \forall\,v\in\V.
\label{eq:fixed-spatial-components}
\end{equation}
Then the unique solution $u(\xi)$ of \eqref{eq:pvp} satisfies, for every $\xi\in\Omega$,
\begin{equation}
u(\xi)=\compactsum_{\ell=0}^{L}\phi_\ell(\xi)w_\ell.
\label{eq:fixed-solution}
\end{equation}
The linear span of the solutions $u(\xi)$ satisfies
\begin{equation}
\operatorname{span}\{u(\xi):\xi\in\Omega\}
\subset
\operatorname{span}\{w_0,\ldots,w_L\},
\label{eq:fixed-solution-span}
\end{equation}
and has dimension at most $L+1$. 
The coefficient function $\eta_N$ belongs to $\operatorname{span}\{\phi_0,\ldots,\phi_L\}$ for every $N\ge1$.
Moreover, with probability one, there is a finite random index $T$ such that $u_N(\xi)=u(\xi)$ for $\mu$-almost every $\xi\in\Omega$ and all $N\ge T$, and at most $L+1$ nonzero directions are generated.
\end{proposition}

\begin{proof}
For an arbitrary $v\in\V$, the defining equations \eqref{eq:fixed-spatial-components} give
\[
\begin{array}{ll}
a\Bigl(
\compactsum_{\ell=0}^{L}\phi_\ell(\xi)w_\ell,v
\Bigr)
&=
\compactsum_{\ell=0}^{L}\phi_\ell(\xi)a(w_\ell,v)=
\compactsum_{\ell=0}^{L}\phi_\ell(\xi)\langle F_\ell,v\rangle=
\langle f(\xi),v\rangle.
\end{array}
\]
Thus $\compactsum_{\ell=0}^{L}\phi_\ell(\xi)w_\ell$ solves \eqref{eq:pvp} and equals $u(\xi)$ by uniqueness, proving \eqref{eq:fixed-solution} and the span inclusion \eqref{eq:fixed-solution-span}.
For $g_N\ne0$, substitute the representation \eqref{eq:fixed-separable-f} and the identity $\langle F_\ell,g_N\rangle=a(w_\ell,g_N)$ from \eqref{eq:fixed-spatial-components} into the coefficient formula \eqref{eq:fixed-eta-basic} from Proposition~\ref{prop:fixed-orthogonal} to obtain
\[
\begin{array}{ll}
\eta_N(\xi)
&=
\frac{\langle f(\xi),g_N\rangle}{a(g_N,g_N)}=
\compactsum_{\ell=0}^{L}\phi_\ell(\xi)
\frac{\langle F_\ell,g_N\rangle}{a(g_N,g_N)}=
\compactsum_{\ell=0}^{L}\phi_\ell(\xi)
\frac{a(w_\ell,g_N)}{a(g_N,g_N)},
\end{array}
\]
which shows that the coefficient function belongs to $\operatorname{span}\{\phi_0,\ldots,\phi_L\}$.
When $g_N=0$, Algorithm~\ref{alg:VS} sets $\eta_N\equiv0$, so the same inclusion holds.

The essential solution span $W_*$ lies in $\operatorname{span}\{w_0,\ldots,w_L\}$ by \eqref{eq:fixed-solution-span}, so $r_*\le L+1$.
Proposition~\ref{prop:fixed-orthogonal} therefore guarantees exactness for $\mu$-almost every parameter after an almost surely finite number of samples, with at most $L+1$ nonzero directions.
\end{proof}

\section{Numerical experiments}
\label{sec:numerics}

The first experiment examines error decay for a parametric diffusion problem in the setting of Section~\ref{sec:finite-term}, together with sensitivity to discretization and direction-discarding thresholds.
The next two experiments compare five-dimensional solution families for parameter-dependent and parameter-independent operators.
They examine the relation between recovery of the spatial solution span and accuracy of the VS approximation, and illustrate the projection structure established in Proposition~\ref{prop:fixed-orthogonal}.

For each experiment, let $D$ denote the spatial domain specified below and take $\mathcal V=H_0^1(D)$.
To measure the accuracy of VS at a fixed spatial resolution, we compare it with the full-order solution in the same finite-element space $V_h\subset\mathcal V$ on a mesh of size $h$.
Let $A_h(\xi)$ be the assembled stiffness matrix, and let $u_h(\xi)$ and $u_{N,h}(\xi)$ denote the full-order discrete solution and the VS approximation after $N$ iterations, respectively.
When using matrix notation, identify finite-element functions with their nodal coefficient vectors and define the discrete error and its energy by
\[
R_h^{(N)}(\xi):=u_h(\xi)-u_{N,h}(\xi)
\quad\mbox{and}\quad
E_{N,h}(\xi):=
R_h^{(N)}(\xi)^T A_h(\xi)R_h^{(N)}(\xi).
\]
This comparison measures the VS approximation error relative to $u_h(\xi)$ and therefore excludes the spatial discretization error between $u(\xi)$ and $u_h(\xi)$.
All experiments use continuous piecewise linear finite elements and double-precision arithmetic, with specific discretization and evaluation settings as described below.

We approximate the parameter average of $E_{N,h}$ using an evaluation rule $\{\xi^{(i)},\varpi_i\}_{i=1}^{n_{\mathrm{eval}}}$ with $\varpi_i>0$ and $\compactsum_i\varpi_i=1$.
The resulting discrete counterpart of the parameter-averaged error energy in \eqref{eq:error-energy-definitions} is
\[
\cE_{N,h}^{\mathrm{eval}}
:=
\compactsum_{i=1}^{n_{\mathrm{eval}}}
\varpi_iE_{N,h}(\xi^{(i)}).
\]
For an independent test set, the weights are $\varpi_i=1/n_{\mathrm{eval}}$, whereas numerical quadrature uses the corresponding normalized quadrature weights.
Normalizing by the initial value gives the reported relative error energy $\cE_{N,h}^{\mathrm{eval}}/\cE_{0,h}^{\mathrm{eval}}$, whose square root is the relative root-mean-square discrete energy-norm error.

The convergence analysis applies to the VS iteration in exact arithmetic without discarding nonzero directions.
In the numerical experiments, floating-point arithmetic and the direction-discarding thresholds specified below may perturb the updates, and the convergence estimates do not cover these perturbations.

\subsection{Convergence for a parametric diffusion problem}

Consider
\begin{equation*}
\begin{array}{ll}
-\frac{\mathrm d}{\mathrm dx}
\bigl(k(x,\xi)\frac{\mathrm du}{\mathrm dx}(x,\xi)\bigr)=1,
\quad x\in(0,1)
\quad\mbox{and}\quad u(0,\xi)=u(1,\xi)=0,
\end{array}
\end{equation*}
where $\xi=(\xi_1,\xi_2)\in[-1,1]^2$ and
\begin{equation*}
\begin{array}{ll}
k(x,\xi)=2+0.4\xi_1\sin(2\pi x)+0.3\xi_2\cos(\pi x).
\end{array}
\end{equation*}
For every $x\in(0,1)$ and $\xi\in[-1,1]^2$, we have $1.3\le k(x,\xi)\le2.7$.
The diffusion bilinear form is therefore symmetric, uniformly coercive, and uniformly bounded, with affine parameter dependence through continuous coefficient functions.
With the constant right-hand side, this problem is covered by Theorem~\ref{thm:finite-term-main}.

We discretized the spatial problem with continuous piecewise linear finite elements on uniform meshes of $64$, $128$, and $256$ elements, using the $128$-element mesh for the main results and the other two to examine sensitivity to spatial resolution.
On each mesh, the element stiffness matrices were assembled using the exact element averages of the diffusion coefficient, and the constant load was integrated exactly.

At every VS iteration, a fresh parameter was drawn independently from the uniform probability measure on $[-1,1]^2$.
This makes the direction sampling consistent with the continuous parameter distribution that defines the parameter-averaged error energy, without selecting parameters from a fixed training set.
To examine variability between runs, we generated $100$ independent sampling histories, each with $60$ iterations.

Using the assembled load vector $b_h$ for the constant right-hand side $1$, we computed the candidate direction $\widetilde g_{N,h}$ at step $N$ by solving the discrete residual problem
\[
A_h(\zeta_N)\widetilde g_{N,h}
=
b_h-A_h(\zeta_N)u_{N-1,h}(\zeta_N).
\]
For each accepted candidate direction, we evaluated the associated parameter-dependent coefficient function using the finite-element counterpart of \eqref{eq:eta-residual}, assembling its numerator by testing the algebraic residual $b_h-A_h(\xi)u_{N-1,h}(\xi)$ against the accepted direction as in Algorithm~\ref{alg:VS}.
Full-order finite-element solutions supplied the reference values used only for error evaluation.
To approximate the parameter integrals in this evaluation, we used a $32\times32$ tensor-product Gauss--Legendre rule with weights normalized for the uniform probability measure.
On the $128$-element mesh, a $64\times64$ rule was also used to check the parameter-integration error.

In this experiment, a candidate direction was discarded only if $\widetilde g_{N,h}^{T}A_h(\zeta_N)\widetilde g_{N,h} \le 10^{-28}\mathcal E_{0,h}^{\mathrm{eval}}$.
At this threshold, all $60$ directions were accepted in all $100$ histories.

\begin{figure}[htbp]
\centering
\reviewfigure[width=0.72\textwidth]{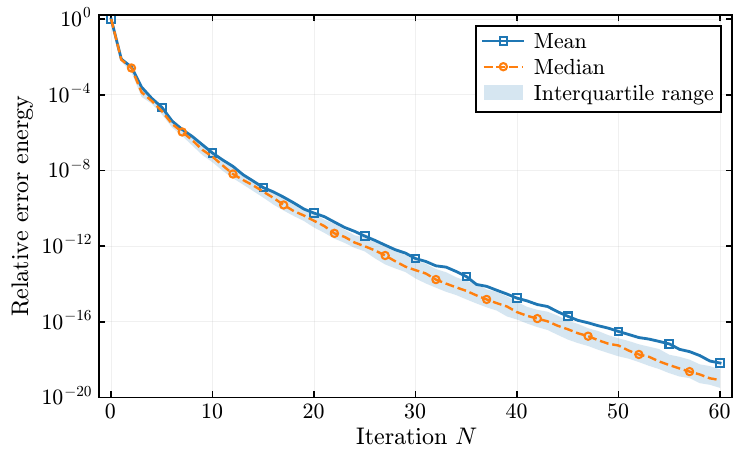}
\caption{Relative error energy for the parameter-dependent diffusion problem on the $128$-element mesh.
The solid and dashed curves show the mean and median over $100$ sampling histories, while the shaded region indicates the interquartile range.}
\label{fig:exp1-test-statistics}
\end{figure}

Figure~\ref{fig:exp1-test-statistics} shows rapid decay of the mean and median of the relative error energy $\cE_{N,h}^{\mathrm{eval}}/\cE_{0,h}^{\mathrm{eval}}$ over the $100$ sampling histories.
On the $128$-element mesh, the mean relative error energy is $8.30\times10^{-8}$, $5.57\times10^{-11}$, $2.19\times10^{-13}$, and $6.63\times10^{-19}$ at steps $10$, $20$, $30$, and $60$, respectively.
The corresponding standard errors are $9.10\times10^{-9}$, $1.33\times10^{-11}$, $6.01\times10^{-14}$, and $1.87\times10^{-19}$.

\begin{figure}[htbp]
\centering
\reviewfigure[width=0.98\textwidth]{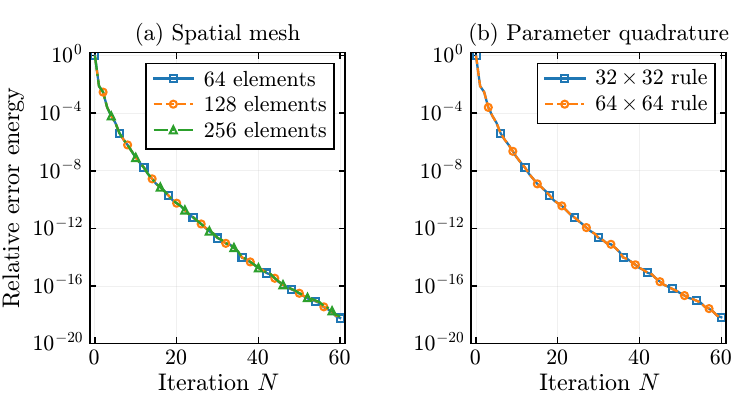}
\caption{Sensitivity of the mean relative error energy to spatial and parameter discretization.
The left panel compares meshes with $64$, $128$, and $256$ elements, while the right panel compares $32\times32$ and $64\times64$ Gauss--Legendre rules on the $128$-element mesh.}
\label{fig:exp1-sensitivity}
\end{figure}

The mean relative error-energy curves for the three spatial meshes in Figure~\ref{fig:exp1-sensitivity} show very similar decay over the computed range.
At step $30$, the mean relative error energies for $64$, $128$, and $256$ elements are $2.12\times10^{-13}$, $2.19\times10^{-13}$, and $2.21\times10^{-13}$, respectively.

To quantify sensitivity to parameter quadrature on the $128$-element mesh, let $\bar q_N^{(32)}$ and $\bar q_N^{(64)}$ denote the relative error energies evaluated using the $32\times32$ and $64\times64$ Gauss--Legendre rules, respectively, and averaged over the $100$ sampling histories.
For the two rules tested, we obtained
\[
\begin{array}{ll}
\max_{0\le N\le60}
\frac{|\bar q_N^{(32)}-\bar q_N^{(64)}|}
{\bar q_N^{(64)}}
=
6.10\times10^{-8}.
\end{array}
\]
Thus, on this mesh and over the reported iterations, the computed mean relative error-energy curve changes little under the tested quadrature refinement.

To check the discrete updates against the one-step energy identity in Proposition~\ref{prop:energy-identity}, define the discrete residual vector by
\[
r_{N,h}(\xi):=b_h-A_h(\xi)u_{N,h}(\xi).
\]
For each accepted step, we computed the normalized defect as
\[
\begin{array}{ll}
\delta_N(\xi)
:=
\frac{1}{\mathcal E_{0,h}^{\mathrm{eval}}}
\Big|
E_{N-1,h}(\xi)-E_{N,h}(\xi)
-
\frac{[\widetilde g_{N,h}^{T}r_{N-1,h}(\xi)]^2}
{\widetilde g_{N,h}^{T}A_h(\xi)\widetilde g_{N,h}}
\Big|.
\end{array}
\]
The maximum value of $\delta_N(\xi)$ over the tested histories, evaluation parameters, and meshes was $4.99\times10^{-13}$.
This small normalized defect is consistent with the energy identity up to floating-point error at the tested parameters.

To test the effect of discarding small directions, we used direction-norm tolerances $10^{-12}$, $10^{-14}$, and $10^{-15}$ with the same first $20$ sampling histories in the $128$-element experiment.
Here the tolerance $\varepsilon$ is relative to $(\mathcal E_{0,h}^{\mathrm{eval}})^{1/2}$ in the energy norm induced by $A_h(\zeta_N)$, so a candidate direction is discarded when $\widetilde g_{N,h}^{T}A_h(\zeta_N)\widetilde g_{N,h}\le\varepsilon^2\mathcal E_{0,h}^{\mathrm{eval}}$.
The three tolerances therefore correspond to energy thresholds $10^{-24}$, $10^{-28}$, and $10^{-30}$ times the initial discrete parameter-averaged error energy $\mathcal E_{0,h}^{\mathrm{eval}}$.
With tolerance $10^{-12}$, between $58$ and $60$ directions were retained, whereas all $60$ directions were retained for the other two tolerances.
For direction-norm tolerances $10^{-12}$, $10^{-14}$, and $10^{-15}$, the mean relative error energies at step $60$ over the $20$ sampling histories were $1.36\times10^{-18}$, $1.20\times10^{-18}$, and $1.20\times10^{-18}$, respectively.
Thus, the threshold affects only the very small terminal error level in this experiment.
We do not use these computations to infer a sharp asymptotic decay rate.

\subsection{Five-dimensional solution families}

The next two experiments provide two structural examples with exactly five-dimensional solution spans.
The first illustrates the behavior of the VS iteration for a parameter-dependent operator, while the second illustrates the orthogonal projection structure obtained when the operator is parameter-independent.
Together, they distinguish between the dimension of the generated spatial space and the accuracy of the resulting VS approximation.

\subsubsection{A parameter-dependent diffusion operator}

Let $D=\bigcup_{j=1}^{5}I_j$, where $I_j=(2j-2,2j-1)$, and consider the problem on each $I_j$ given by
\begin{equation*}
\begin{array}{ll}
-\frac{\mathrm d}{\mathrm dx}
\bigl(k(x,\xi)\frac{\mathrm du}{\mathrm dx}(x,\xi)\bigr)=1
\quad\mbox{and}\quad u|_{\partial I_j}=0,
\end{array}
\end{equation*}
where $\xi=(\xi_1,\ldots,\xi_5)\in[-1,1]^5$ and $k(x,\xi)=1.2+0.8\xi_j$ for $x\in I_j$.
On each component interval, the diffusion coefficient is constant in space, so changing $\xi_j$ rescales the solution without changing its spatial shape.
For $I_j=(x_j^-,x_j^+)$, the solution with unit diffusion and load $1$, extended by zero to the other components, is
\begin{equation*}
\begin{array}{ll}
\psi_j(x)=
\begin{cases}
\tfrac12(x-x_j^-)(x_j^+-x),&x\in I_j,\\
0,&x\notin I_j.
\end{cases}
\end{array}
\end{equation*}
Dividing each mode by its diffusion coefficient therefore gives the solution
\begin{equation}
u(x,\xi)=\compactsum_{j=1}^{5}\frac{1}{1.2+0.8\xi_j}\psi_j(x).
\label{eq:exp2-exact-solution}
\end{equation}
The nonzero functions $\psi_j$ have disjoint supports and are therefore linearly independent.
Each coordinate $\xi_j$ can vary while the others remain fixed, which also makes the five parameter coefficient functions linearly independent.
In particular, the difference between solutions at two parameter values that differ only in $\xi_j$ is a nonzero multiple of $\psi_j$.
Thus every $\psi_j$ belongs to the span of the solutions $u(\xi)$, and the representation \eqref{eq:exp2-exact-solution} shows that this span is exactly five-dimensional.

Each component interval had $30$ interior degrees of freedom ($h=1/31$), giving $150$ in total.
On these meshes, the constant-coefficient stiffness matrices and constant load vectors were assembled exactly.
To express the resulting discrete family in the same modal form, let $\psi_{j,h}$ be the finite-element solution for unit diffusion and load $1$ on $I_j$, extended by zero to the other components.
The full-order discrete solution then has the expansion \eqref{eq:exp2-exact-solution} with $\psi_j$ replaced by $\psi_{j,h}$, so its solution span is also exactly five-dimensional.

\subsubsection{A parameter-independent Laplace operator}

On $D=(0,1)$, consider
\begin{equation*}
\begin{array}{ll}
-u''(x,\xi)=f_{\mathrm{vol}}(x,\xi)
\quad\mbox{and}\quad u(0,\xi)=u(1,\xi)=0,
\end{array}
\end{equation*}
where $\xi=(\xi_1,\ldots,\xi_5)\in[-1,1]^5$ and
\begin{equation*}
\begin{array}{ll}
f_{\mathrm{vol}}(x,\xi)=\compactsum_{j=1}^{5}\xi_j(j\pi)^2\sin(j\pi x).
\end{array}
\end{equation*}
This density induces the functional $f(\xi)\in\Vp$ in \eqref{eq:pvp} through
\[
\langle f(\xi),v\rangle=\int_0^1 f_{\mathrm{vol}}(x,\xi)v(x)\,\mathrm dx
\qquad \forall\,v\in\V.
\]
For $j=1,\ldots,5$, the right-hand side $(j\pi)^2\sin(j\pi x)$ has solution $\sin(j\pi x)$ with homogeneous Dirichlet boundary conditions, so Proposition~\ref{prop:fixed-separable} gives
\begin{equation*}
\begin{array}{ll}
u(x,\xi)=\compactsum_{j=1}^{5}\xi_j\sin(j\pi x),
\end{array}
\end{equation*}
again with a five-dimensional solution span.
All parameters share the energy inner product
\begin{equation*}
\begin{array}{ll}
a(v,w)=\int_0^1v'(x)w'(x)\,\mathrm dx.
\end{array}
\end{equation*}
The mesh had $100$ interior degrees of freedom ($h=1/101$).
To construct the discrete counterparts of the five modes, we assembled the stiffness matrix exactly and each load vector corresponding to a sinusoidal right-hand side using three-point Gauss--Legendre quadrature on every element.
With $\psi_{j,h}$ denoting the finite-element solution corresponding to the load $(j\pi)^2\sin(j\pi x)$, the discrete solution takes the form $u_h(\xi)=\compactsum_{j=1}^{5}\xi_j\psi_{j,h}$, and its solution span is five-dimensional.

In both experiments, parameters were sampled independently from the uniform distribution on $[-1,1]^5$, and a fixed independent test set of $5000$ samples was used for error evaluation.
To examine variability between runs, we generated $40$ independent sampling histories with $30$ iterations each.
The two experiments reused the same test parameters and the same sampling sequences for corresponding histories, allowing the problems to be compared with the random inputs held fixed.

Within the two fixed five-mode coordinate systems, we used problem-specific absolute energy thresholds to identify numerically negligible directions.
An unnormalized candidate direction $\widetilde g_{N,h}$ was discarded when $\widetilde g_{N,h}^{T}A_h(\zeta_N)\widetilde g_{N,h}$ was at most $10^{-28}$ for the parameter-dependent diffusion problem or $10^{-24}$ for the parameter-independent Laplace problem.
When a candidate direction was discarded, the approximation remained $u_{N,h}=u_{N-1,h}$, while the iteration count still advanced.

For each of the two problems, we precomputed its fixed modes $\psi_{j,h}$ using NumPy's direct linear solver and evaluated VS updates and energy products in the corresponding five-mode coordinates.
At each newly sampled parameter $\zeta_N$, the coefficients of $R_h^{(N-1)}(\zeta_N)$ were obtained by applying the previously retained VS corrections, in their original order, to the coefficients of the full-order solution $u_h(\zeta_N)$.
Because the full-order solution and every retained direction lie in the span of these modes, the coordinate updates reproduce the same discrete VS iteration in exact arithmetic, without additional solution-space truncation.

To distinguish numerical recovery of the discrete solution space from accuracy of the VS approximation, we estimate the dimension of the retained-direction span in the fixed five-mode coordinates, compare it with the dimension of $\operatorname{span}\{u_h(\xi):\xi\in[-1,1]^5\}$, and track the relative error energy $\cE_{N,h}^{\mathrm{eval}}/\cE_{0,h}^{\mathrm{eval}}$.
For each retained direction, write
\[
g_{k,h}=\compactsum_{j=1}^{5}c_{jk}\psi_{j,h},
\]
and denote its coefficient vector by $\mathbf c_k=(c_{1k},\ldots,c_{5k})^T$.
At iteration $N$, the directions included in this diagnostic are specified by the retained-index set
\[
\mathcal I_N:=\{k\in\{1,\ldots,N\}:g_{k,h}\text{ is retained}\}.
\]
To compare the retained directions on a common scale, for $k\in\mathcal I_N$ set
\[
\widehat{\mathbf c}_k:=\frac{\mathbf c_k}{\|\mathbf c_k\|_2}
\quad\mbox{and}\quad
\widehat C_N:=[\widehat{\mathbf c}_k]_{k\in\mathcal I_N},
\]
with the columns ordered by increasing $k$.
If $\mathcal I_N=\emptyset$, set $r_N^{\mathrm{num}}:=0$, and otherwise estimate the dimension by counting singular values above a relative threshold according to
\[
r_N^{\mathrm{num}}
:=
\#\Bigl\{j:
\sigma_j(\widehat C_N)>
10^{-12}\sigma_1(\widehat C_N)
\Bigr\}.
\]
Here $\sigma_j(\widehat C_N)$ denotes the $j$th singular value in nonincreasing order.
The resulting numerical rank $r_N^{\mathrm{num}}$ is used only as a diagnostic and does not affect the VS iteration or the retention of nonzero directions.
At step $5$, the minimum value of $\sigma_5(\widehat C_5)/\sigma_1(\widehat C_5)$ over the $40$ histories is $7.75\times10^{-3}$ for the parameter-dependent problem and $1.91\times10^{-1}$ for the parameter-independent problem, both well above the prescribed rank tolerance.

\begin{figure}[htbp]
\centering
\reviewfigure[width=0.7\textwidth]{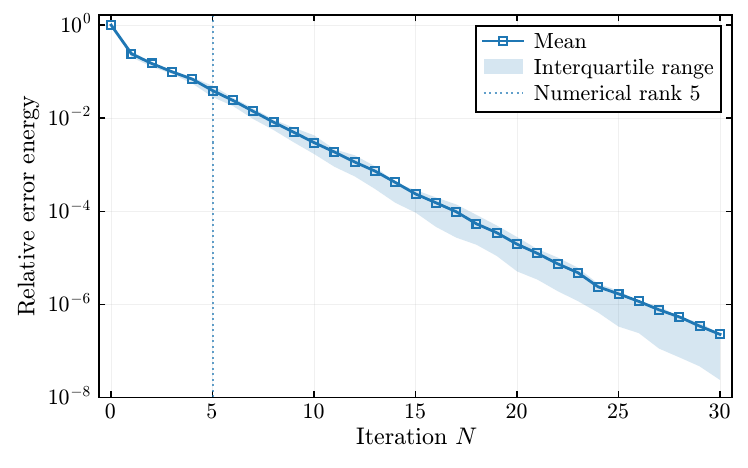}
\caption{Relative test-set error energy for the parameter-dependent five-dimensional problem.
The curve and shading show the mean and interquartile range over $40$ sampling histories, while the dotted line marks step $5$, when the numerical rank reaches five in every history.}
\label{fig:exp2-pde-energy-decay}
\end{figure}

\begin{figure}[htbp]
\centering
\reviewfigure[width=0.7\textwidth]{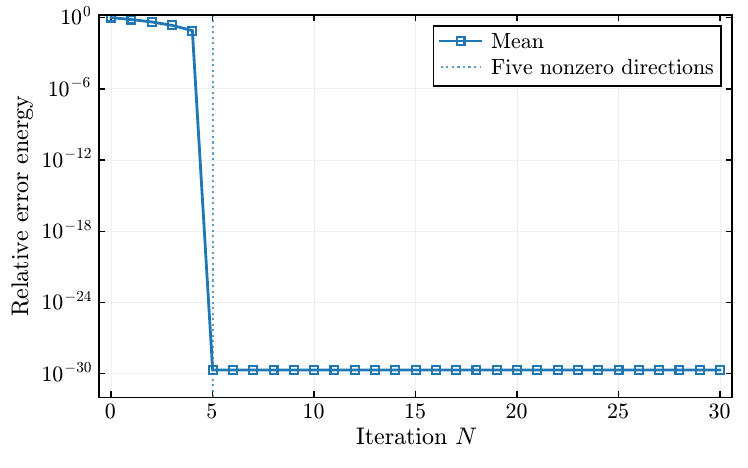}
\caption{Mean relative test-set error energy over $40$ sampling histories for the parameter-independent five-dimensional problem.
The dotted line marks recovery to roundoff after five nonzero directions.}
\label{fig:exp3-pde-energy}
\end{figure}

\begin{table}[htbp]
\centering
\caption{Mean relative test-set error energies over $40$ sampling histories for two problems that both satisfy $r_5^{\mathrm{num}}=5$ in every history.}
\label{tab:five-dimensional-comparison}
\scriptsize
\resizebox{\textwidth}{!}{%
\begin{tabular}{lcccc}
\toprule
Problem & \shortstack{Parameter-dependent\\operator} & $r_5^{\mathrm{num}}$ & $\cE_{5,h}^{\mathrm{eval}}/\cE_{0,h}^{\mathrm{eval}}$ & $\cE_{30,h}^{\mathrm{eval}}/\cE_{0,h}^{\mathrm{eval}}$\\
\midrule
Diffusion & Yes & $5$ & $3.922\times10^{-2}$ & $2.27\times10^{-7}$\\
Laplace & No & $5$ & $2.056\times10^{-30}$ & $2.056\times10^{-30}$\\
\bottomrule
\end{tabular}%
}
\end{table}

Figure~\ref{fig:exp2-pde-energy-decay} and Table~\ref{tab:five-dimensional-comparison} show that, for the parameter-dependent operator, the generated directions reach numerical rank five at step $5$ in every history, while the mean relative error energy is $3.922\times10^{-2}$.
The mean relative error energy then continues to decrease within the same five-dimensional spatial space, reaching $2.27\times10^{-7}$ at step $30$.
This illustrates that, for a parameter-dependent operator, additional VS terms may continue to improve the parameter-dependent representation even after the generated spatial space has reached the dimension of the discrete solution family.

For the parameter-independent operator, Figure~\ref{fig:exp3-pde-energy} shows that the mean of the relative test-set error energy $\cE_{N,h}^{\mathrm{eval}}/\cE_{0,h}^{\mathrm{eval}}$ decreases to a level consistent with floating-point roundoff after five nonzero directions.
To quantify orthogonality in the discrete energy inner product, write $A_h$ for the common parameter-independent stiffness matrix and form $\Phi_h$ with the nodal vectors of the Laplace modes $\psi_{1,h},\ldots,\psi_{5,h}$ as its columns in that order, so that $H_h=\Phi_h^TA_h\Phi_h$ is the energy Gram matrix in these coordinates.
Thus $\mathbf c^T H_h\mathbf d$ is the discrete energy inner product of the directions represented by coefficient vectors $\mathbf c$ and $\mathbf d$.
To measure the deviation of the five retained directions from orthogonality, denote their coefficient vectors in history $s$ by $\mathbf c_i^{(s)}$ and define the maximum absolute normalized energy inner product by
\[
\varepsilon_\perp
:=
\max_{\substack{1\le s\le40\\1\le i<j\le5}}
\tfrac{|(\mathbf c_i^{(s)})^T H_h\mathbf c_j^{(s)}|}
{\sqrt{(\mathbf c_i^{(s)})^T H_h\mathbf c_i^{(s)}}
 \sqrt{(\mathbf c_j^{(s)})^T H_h\mathbf c_j^{(s)}}}.
\]
The resulting overall maximum was $\varepsilon_{\perp}=1.77\times10^{-14}$.
The mean relative error energy was $2.056\times10^{-30}$ at step $5$ and remained unchanged through step $30$.
The small value of $\varepsilon_\perp$ is consistent with the direction orthogonality in Proposition~\ref{prop:fixed-orthogonal}, while recovery to roundoff after five nonzero directions agrees with the finite-recovery conclusion for separated right-hand sides in Proposition~\ref{prop:fixed-separable}.

\section{Conclusions}
\label{sec:conclusion}
Although each VS direction is constructed from a residual problem at a single
sampled parameter, the resulting approximations converge almost surely to the
full parametric solution in the parameter-averaged energy norm under the
finite-dimensional approximation conditions on the error family in
Theorem~\ref{thm:abstract-main}. 
Under the additional quantitative conditions
\eqref{eq:abstract-rate-assumptions}, the expected parameter-averaged
error energy after $N$ iterations has an $O((\log N)^{-1})$ upper bound.
For operators and right-hand sides with finite-term separated
representations, the required approximation conditions are verified in
Theorem~\ref{thm:finite-term-main} without any additional continuity
assumption on the coefficient functions in the operator representation.
When the exact solutions lie in a fixed finite-dimensional subspace for
$\mu$-almost every parameter, the expected parameter-averaged error energy
satisfies the sharper $O(N^{-1})$ bound of
Theorem~\ref{thm:finite-dimensional}. 
For a parameter-independent operator, Proposition~\ref{prop:fixed-orthogonal}
shows that each VS approximation is the energy-orthogonal projection of the
exact solution onto the space generated by the VS directions and satisfies
the best-approximation property in the common energy norm.
If the essential solution span $W_*$ is also finite-dimensional, the
iteration recovers $W_*$ after an almost surely finite number of iterations
and reproduces the exact solution for $\mu$-almost every parameter from
that iteration onward.
The numerical experiments illustrate convergence for parameter-dependent
problems and finite-step recovery up to roundoff for parameter-independent
problems with finite-dimensional solution families.
The observed error-energy decay can be substantially faster than the
general theoretical upper bounds suggest.

\appendix

\section{Measurability of the iteration and generated spaces}
\label{app:measurability}

This appendix verifies the measurability needed to treat the VS error energies as random variables and to take the conditional expectations used in Section~\ref{sec:convergence}.
We establish measurability of the solution map and the iterates to support the conditional energy calculation in Proposition~\ref{prop:energy-identity}, and then construct the measurable projections used in the proof of Theorem~\ref{thm:finite-term-main}.

\begin{lemma}
\label{lem:solution-measurability}
Let $u(\xi)$ be the unique solution of \eqref{eq:pvp}.
The map $(\xi,v,w)\mapsto a(v,w;\xi)$ is jointly measurable, and the solution map $u:\Omega\to\V$ is strongly measurable.
\end{lemma}

\begin{proof}
For fixed $\xi$ and vectors $v,v',w,w'\in\V$, apply \eqref{eq:bounded} to $a(v-v',w;\xi)$ and $a(v',w-w';\xi)$, whose sum is $a(v,w;\xi)-a(v',w';\xi)$.
Both terms tend to zero as $(v',w')\to(v,w)$, so $(v,w)\mapsto a(v,w;\xi)$ is continuous.
For each fixed pair $(v,w)$, the map $\xi\mapsto a(v,w;\xi)$ is Borel measurable.
The Carath\'eodory joint measurability theorem on the separable space $\V\times\V$ therefore applies, so $(\xi,v,w)\mapsto a(v,w;\xi)$ is jointly measurable.

If $\V$ is finite-dimensional, choose an orthonormal basis $e_1,\ldots,e_r$ of $\V$, where $r=\dim\V$.
In this basis, the coefficient vector $\mathbf c_r(\xi)$ of the solution satisfies $A_r(\xi)\mathbf c_r(\xi)=b_r(\xi)$, where $(A_r(\xi))_{ij}=a(e_j,e_i;\xi)$ and $(b_r(\xi))_i=\langle f(\xi),e_i\rangle$.
Because the basis vectors are fixed, the entries $a(e_j,e_i;\xi)$ and $\langle f(\xi),e_i\rangle$ are Borel measurable functions of $\xi$.
For every $\mathbf z=(z_1,\ldots,z_r)^T\in\mathbb R^r$, \eqref{eq:coercive} gives $\mathbf z^TA_r(\xi)\mathbf z=a(\compactsum_jz_je_j,\compactsum_jz_je_j;\xi)\ge\alpha\|\mathbf z\|_2^2$, so $A_r(\xi)\mathbf z=0$ forces $\mathbf z=0$, and $A_r(\xi)$ is therefore invertible.
Since $(A,b)\mapsto A^{-1}b$ is continuous for invertible $A$, the composition $\xi\mapsto A_r(\xi)^{-1}b_r(\xi)=\mathbf c_r(\xi)$ is measurable, and $u(\xi)=\compactsum_{j=1}^r(\mathbf c_r(\xi))_je_j$ is strongly measurable.

In the infinite-dimensional case, we apply the preceding matrix argument on Galerkin spaces and then pass to the exact solution by approximation.
Separability provides a countable orthonormal basis $\{e_j\}_{j\ge1}$ of $\V$, whose finite spans $V_m:=\spanop\{e_1,\ldots,e_m\}$ have dense union in $\V$.
For each $\xi$, let $u_m^{\mathrm{Gal}}(\xi)\in V_m$ be the Galerkin solution satisfying $a(u_m^{\mathrm{Gal}}(\xi),v;\xi)=\langle f(\xi),v\rangle$ for every $v\in V_m$.
In the chosen basis, this equation becomes $A_m(\xi)\mathbf c_m(\xi)=b_m(\xi)$, where $(A_m(\xi))_{ij}=a(e_j,e_i;\xi)$ and $(b_m(\xi))_i=\langle f(\xi),e_i\rangle$.
As in the finite-dimensional case, $A_m(\xi)$ is invertible because $\mathbf z^TA_m(\xi)\mathbf z\ge\alpha\|\mathbf z\|_2^2$ for every $\mathbf z\in\mathbb R^m$.
The coefficient map $\xi\mapsto\mathbf c_m(\xi)=A_m(\xi)^{-1}b_m(\xi)$ is a composition of the Borel measurable map $\xi\mapsto(A_m(\xi),b_m(\xi))$ and the continuous map $(A,b)\mapsto A^{-1}b$, so the finite expansion $u_m^{\mathrm{Gal}}(\xi)=\compactsum_{j=1}^m(\mathbf c_m(\xi))_je_j$ is strongly measurable.
To prove that $u$ is strongly measurable by passing to the limit of $u_m^{\mathrm{Gal}}$, apply C\'ea's lemma with the bounds \eqref{eq:coercive}--\eqref{eq:bounded} to obtain
\begin{equation}
\begin{array}{ll}
\|u(\xi)-u_m^{\mathrm{Gal}}(\xi)\|_{\V}
\le \frac{M}{\alpha}
\inf_{v\in V_m}\|u(\xi)-v\|_{\V}.
\end{array}
\label{eq:galerkin-cea-bound}
\end{equation}
For each fixed $\xi$, the increasing spaces $V_m$ have dense union in $\V$, so $\inf_{v\in V_m}\|u(\xi)-v\|_{\V}\to0$ and $u_m^{\mathrm{Gal}}(\xi)\to u(\xi)$ in $\V$ by \eqref{eq:galerkin-cea-bound}.
Thus $u$ is strongly measurable as the pointwise limit of the strongly measurable maps $u_m^{\mathrm{Gal}}$.
\end{proof}

\begin{proposition}
\label{prop:iteration-measurability}
Let $\{u_N\}$, $\{g_N\}$, and $\{\eta_N\}$ be generated by Algorithm~\ref{alg:VS} for the problem \eqref{eq:pvp}.
Let $R^{(N)}$, $E_N$, and $\cE_N$ be defined by \eqref{eq:error-definition} and \eqref{eq:error-energy-definitions}.
Then, for every $N\ge0$, the map $R^{(N)}:\Omega\times\Theta\to\V$ is strongly measurable with respect to $\mathcal B\otimes\F_N$, the map $E_N:\Omega\times\Theta\to\mathbb R$ is measurable with respect to the same product $\sigma$-algebra, and $\cE_N:\Theta\to\mathbb R$ is $\F_N$-measurable.
For every $N\ge1$, $g_N:\Theta\to\V$ is $\F_N$-measurable. 
\end{proposition}

\begin{proof}
To establish joint measurability while keeping track of the sampled coordinates, we prove by induction that $R^{(N)}$ is strongly measurable with respect to $\mathcal B\otimes\F_N$.
At $N=0$, Lemma~\ref{lem:solution-measurability} shows that $u:\Omega\to\V$ is strongly measurable, so $R^{(0)}(\xi,\omega)=u(\xi)$ is strongly measurable with respect to $\mathcal B\otimes\F_0$ because it does not depend on the sampling history.
For the induction step, assume that $R^{(N-1)}$ is strongly measurable with respect to $\mathcal B\otimes\F_{N-1}$ and consider the sampling map $\omega\longmapsto(\Xi_N(\omega),\omega)$.
The first coordinate $\Xi_N$ is $\F_N$-measurable, while the second coordinate $\omega\mapsto\omega$ is measurable from $(\Theta,\F_N)$ to $(\Theta,\F_{N-1})$ because $\F_{N-1}\subset\F_N$.
The sampling map is therefore measurable from $(\Theta,\F_N)$ to $(\Omega\times\Theta,\mathcal B\otimes\F_{N-1})$.
The direction identity \eqref{eq:direction-as-error} from Lemma~\ref{lem:basic-error-relations} expresses $g_N(\omega)=R^{(N-1)}(\Xi_N(\omega),\omega)$ as the composition of $R^{(N-1)}$ with the sampling map, so $g_N$ is $\F_N$-measurable.
For the coefficient update, both $(\xi,\omega)\mapsto R^{(N-1)}(\xi,\omega)$ and $(\xi,\omega)\mapsto g_N(\omega)$ are measurable with respect to $\mathcal B\otimes\F_N$.
Composing these maps with the jointly measurable map $(\xi,v,w)\mapsto a(v,w;\xi)$ shows that $a(R^{(N-1)}(\xi,\omega),g_N(\omega);\xi)$ and $a(g_N(\omega),g_N(\omega);\xi)$, the numerator and denominator in \eqref{eq:eta-residual}, are measurable with respect to the same product $\sigma$-algebra.
Where $g_N(\omega)\ne0$, \eqref{eq:coercive} gives $a(g_N(\omega),g_N(\omega);\xi)\ge\alpha\|g_N(\omega)\|_{\V}^2>0$, so the quotient in \eqref{eq:eta-residual} is measurable there.
The set $\{g_N=0\}$ is measurable because $g_N$ is measurable, and the algorithm sets $\eta_N=0$ on this set, giving a $\mathcal B\otimes\F_N$-measurable function on all of $\Omega\times\Theta$.
The update \eqref{eq:error-update}, i.e., $R^{(N)}=R^{(N-1)}-\eta_Ng_N$, preserves strong measurability with respect to $\mathcal B\otimes\F_N$ and closes the induction.
The pointwise energy $E_N(\xi,\omega)=a(R^{(N)}(\xi,\omega),R^{(N)}(\xi,\omega);\xi)$ in \eqref{eq:error-energy-definitions} is $\mathcal B\otimes\F_N$-measurable as a composition of $R^{(N)}$ with the jointly measurable bilinear form.
Because $E_N$ is nonnegative, the parameter integral $\cE_N(\omega)=\int_\Omega E_N(\xi,\omega)\dd\mu(\xi)$ is then $\F_N$-measurable.
\end{proof}

It remains to verify the measurability of the $\V$-orthogonal projections $P_N$ onto the random approximation spaces $Z_N$ constructed in the proof of Theorem~\ref{thm:finite-term-main}.
Since these spaces contain the sampled directions, we construct their orthonormal bases by a measurable Gram--Schmidt procedure that also allows vanishing or linearly dependent input vectors.

\begin{lemma}
\label{lem:measurable-GS}
Let $m\ge1$ and let $h_1,\ldots,h_m:\Theta\to\V$ be strongly measurable random vectors.
Then there are strongly measurable random vectors $e_1,\ldots,e_m:\Theta\to\V$ such that, for every $\omega\in\Theta$, the nonzero vectors among $e_1(\omega),\ldots,e_m(\omega)$ form an orthonormal basis of $\spanop\{h_1(\omega),\ldots,h_m(\omega)\}$.
\end{lemma}

\begin{proof}
At step $j$ of the Gram--Schmidt procedure, subtract the previously constructed projections from $h_j$ to obtain the remainder $h_j-\compactsum_{i<j}(h_j,e_i)_{\V}e_i$, where the terms with $e_i=0$ vanish.
Each scalar coefficient $(h_j,e_i)_{\V}$ is measurable, so this finite linear combination of strongly measurable vectors is strongly measurable.
Since its norm is measurable, $e_j$ is defined on two measurable sets by dividing the remainder by its norm where it is nonzero and setting $e_j=0$ otherwise.
The normalization map $v\mapsto v/\|v\|_{\V}$ is continuous on $\V\setminus\{0\}$, so $e_j$ is strongly measurable on both sets and hence on $\Theta$.
\end{proof}

For each fixed $N$, the space $Z_N=W_{p_N}+G_N$ constructed in the proof of Theorem~\ref{thm:finite-term-main} combines the deterministic approximation subspace $W_{p_N}\subset\V$ from Proposition~\ref{prop:solution-approx} with the random direction space $G_N\subset\V$ in \eqref{eq:generated-direction-space}, i.e., $G_N=\operatorname{span}\{g_1,\ldots,g_N\}$.
Here $p_N$ is the integer in \eqref{eq:pN-choice}, i.e., $p_N=\lceil\log_2(N+2)\rceil$.
Choose an orthonormal basis of the deterministic space $W_{p_N}$ independently of the sampling history, so that this basis together with $g_1,\ldots,g_N$ forms a finite list of $\F_N$-measurable vectors spanning $Z_N$.
For a nonempty list, the vectors $e_j:\Theta\to\V$ from Lemma~\ref{lem:measurable-GS} are $\F_N$-measurable, and their nonzero members form an orthonormal basis of $Z_N$ for every sampling history.
If the list is empty, $Z_N=\{0\}$ and no basis vectors are needed.

The resulting $\V$-orthogonal projection is $P_N(\omega)v=\compactsum_j(v,e_j(\omega))_{\V}e_j(\omega)$, where an empty sum is understood to be zero.
For every fixed $v\in\V$, this is a finite sum of $\F_N$-measurable vectors, so the projection $P_N$ satisfies \eqref{eq:projection-measurability}.
The projected error and remainder in \eqref{eq:XY-decomposition}, i.e., $X_N=P_NR^{(N)}$ and $Y_N=(I-P_N)R^{(N)}$, are therefore well defined.

\section{Dimension bounds for local inverse expansions}
\label{app:neumann}

The dimension estimate in Proposition~\ref{prop:solution-approx} counts parameter monomials rather than ordered products of the component operators.
We justify this count without assuming that the component operators commute, and record the truncation bound for the same local inverse expansion.

To isolate the scalar parameter dependence in the local inverse expansion, we start from the separated representations \eqref{eq:separable-a}--\eqref{eq:separable-f} and the operators $A(\xi):\V\to\Vp$ and $A_q:\V\to\Vp$ in \eqref{eq:operator-definitions}.
At a fixed $\zeta\in\Omega$, the proof of Proposition~\ref{prop:solution-approx} uses the operators $B_q^{(\zeta)}\in\mathcal L(\V)$ and vectors $b_\ell^{(\zeta)}\in\V$, i.e., $B_q^{(\zeta)}=A(\zeta)^{-1}A_q$ and $b_\ell^{(\zeta)}=A(\zeta)^{-1}F_\ell$, for $q=1,\ldots,Q$ and $\ell=0,\ldots,L$.
With $B_q:=B_q^{(\zeta)}$ and $s_q(\xi):=\theta_q(\xi)-\theta_q(\zeta)$, all parameter dependence in the perturbation operator $T_\zeta(\xi)\in\mathcal L(\V)$ lies in the scalar coefficients of $T_\zeta(\xi)=\compactsum_{q=1}^{Q}s_q(\xi)B_q$.
Here $T_\zeta(\xi)$ is the operator introduced in the proof of Proposition~\ref{prop:solution-approx}, i.e., $T_\zeta(\xi)=A(\zeta)^{-1}(A(\xi)-A(\zeta))$.
For $k\ge1$, we have
\begin{equation*}
\begin{array}{ll}
T_\zeta(\xi)^k
=\compactsum_{i_1,\ldots,i_k=1}^{Q}
\bigl(\prod_{j=1}^{k}s_{i_j}(\xi)\bigr)
B_{i_1}\cdots B_{i_k}.
\end{array}
\end{equation*}
In the expansion of $T_\zeta(\xi)^k$, the scalar coefficient $\prod_{j=1}^k s_{i_j}(\xi)$ depends only on how often each index occurs, so equal monomials can be collected without changing the order of the factors in $B_{i_1}\cdots B_{i_k}$.
Write $s(\xi):=(s_1(\xi),\ldots,s_Q(\xi))$, let a multi-index $\nu=(\nu_1,\ldots,\nu_Q)\in\{0,1,\ldots\}^Q$ with $|\nu|=\nu_1+\cdots+\nu_Q=k$ record these multiplicities, and set $s(\xi)^\nu:=\prod_{q=1}^{Q}s_q(\xi)^{\nu_q}$.
Let $\mathcal W(\nu)$ be the set of words over $\{1,\ldots,Q\}$ in which the letter $q$ occurs exactly $\nu_q$ times, so all words in this set have the same scalar coefficient $s(\xi)^\nu$.
For $\ell=0,\ldots,L$, the coefficient of $s(\xi)^\nu$ in $T_\zeta(\xi)^k b_\ell^{(\zeta)}$ is $v_{\zeta,\nu,\ell}=\compactsum_{(i_1,\ldots,i_k)\in\mathcal W(\nu)} B_{i_1}\cdots B_{i_k}b_\ell^{(\zeta)}$, as introduced in the proof of Proposition~\ref{prop:solution-approx}.
The case $k=0$ corresponds to $\nu=0$ and $\mathcal W(0)$ containing only the empty word, so $v_{\zeta,0,\ell}=b_\ell^{(\zeta)}$.
The expansion in the coefficient vectors $v_{\zeta,\nu,\ell}$ is therefore
\[
\begin{array}{ll}
T_\zeta(\xi)^k b_\ell^{(\zeta)}
=\compactsum_{|\nu|=k}s(\xi)^\nu v_{\zeta,\nu,\ell}.
\end{array}
\]
Since the number of required monomials of degree at most $p$ is $\compactsum_{k=0}^{p}\binom{Q+k-1}{k}=\binom{Q+p}{p}$, accounting for the $L+1$ right-hand-side components and the $K$ reference parameters gives the dimension estimate in Proposition~\ref{prop:solution-approx}.
The operator order affects the coefficient vectors but introduces no additional monomials, so this count does not require commutativity.

The remaining question is how accurately $W_p$ approximates $u(\xi)$, which depends on the omitted terms of the inverse expansion.
The Neumann-series argument in the proof of Proposition~\ref{prop:solution-approx} shows that whenever $\|T_\zeta(\xi)\|_{\mathcal L(\V)}\le\rho<1$,
\[
\begin{array}{ll}
\bigl\|(I+T_\zeta(\xi))^{-1}
-\compactsum_{k=0}^{p}(-T_\zeta(\xi))^k\bigr\|_{\mathcal L(\V)}
\le\frac{\rho^{p+1}}{1-\rho}.
\end{array}
\]
For the local expansions \eqref{eq:local-inverse-expansion} in the proof of Proposition~\ref{prop:solution-approx}, $\rho=1/2$ and $\|\mathcal R_{\zeta,p}(\xi)\|_{\mathcal L(\V)}\le2^{-p}$.
The vector $A(\zeta)^{-1}f(\xi)$ has norm at most $F/\alpha$, since $\|A(\zeta)^{-1}\|_{\mathcal L(\Vp,\V)}\le1/\alpha$ by \eqref{eq:coercive} and $\|f(\xi)\|_{\Vp}\le F$ by \eqref{eq:fbound}.
The approximation error $\mathcal R_{\zeta,p}(\xi)A(\zeta)^{-1}f(\xi)$ in $\V$ therefore satisfies
\[
\begin{array}{ll}
\bigl\|
\bigl[
(I+T_\zeta(\xi))^{-1}
-\compactsum_{k=0}^{p}(-T_\zeta(\xi))^k
\bigr]
A(\zeta)^{-1}f(\xi)
\bigr\|_{\V}
\le
\frac{F}{\alpha}2^{-p}.
\end{array}
\]

\section{Constants in the correlation bound}
\label{app:tail}

This appendix specifies the threshold constant $C_0$ in \eqref{eq:above-tail} used in Proposition~\ref{prop:tail-or-decrease}.
For this comparison, recall the parameter-averaged error energy $\cE_N$ in \eqref{eq:error-energy-definitions}, the dimension bound $D_N$ and approximation tolerance $\tau_N$ in \eqref{eq:abstract-approx}, and the error components $X_N,Y_N$ in \eqref{eq:XY-decomposition}, i.e., $R^{(N)}=X_N+Y_N$.
The proof of Proposition~\ref{prop:tail-or-decrease} measures $X_N$ by its squared $L^2(\Omega,\mu;\V)$ norm $S_N=\int_\Omega\|X_N(\xi)\|_{\V}^2\dd\mu(\xi)$.
To ensure that $2M\sqrt{S_N}\tau_N+M\tau_N^2\le\alpha S_N/(2\sqrt{D_N})$ in \eqref{eq:tail-subtraction}, first bound $S_N$ below in terms of $\cE_N$.
Since the claimed correlation bound is immediate when $\cE_N=0$, restrict attention to $\cE_N>0$ under the threshold condition \eqref{eq:above-tail}, i.e., $\cE_N\ge C_0D_N\tau_N^2$.
For the comparison, the upper bound $\cE_N\le2M(S_N+\tau_N^2)$ established in the proof of Proposition~\ref{prop:tail-or-decrease}, together with $D_N\ge1$ and $C_0\ge4M$, gives
\[
\begin{array}{ll}
S_N\ge\frac{\cE_N}{2M}-\tau_N^2
\ge\frac{\cE_N}{2M}-\frac{\cE_N}{C_0D_N}
\ge\frac{\cE_N}{4M}>0.
\end{array}
\]
The $L^2(\mu\otimes\mu)$-norm of $a(X_N(\xi),X_N(\zeta);\xi)$ is at least $\mathcal A_N:=\alpha S_N/\sqrt{D_N}>0$ by \eqref{eq:main-correlation-norm}.
For the other terms in the expansion using \eqref{eq:XY-decomposition}, the proof of Proposition~\ref{prop:tail-or-decrease} bounds the $L^2(\mu\otimes\mu)$-norm of each of $a(X_N(\xi),Y_N(\zeta);\xi)$ and $a(Y_N(\xi),X_N(\zeta);\xi)$ by $M\sqrt{S_N}\tau_N$, and that of $a(Y_N(\xi),Y_N(\zeta);\xi)$ by $M\tau_N^2$.
These three norm bounds add up to $\mathcal B_N:=2M\sqrt{S_N}\tau_N+M\tau_N^2$.
Since $S_N\ge\cE_N/(4M)$ by \eqref{eq:SN-lower} and $\cE_N\ge C_0D_N\tau_N^2$ by \eqref{eq:above-tail}, $D_N\tau_N^2/S_N\le4M/C_0$.
The ratios of the two terms $2M\sqrt{S_N}\tau_N$ and $M\tau_N^2$ in $\mathcal B_N$ to $\mathcal A_N=\alpha S_N/\sqrt{D_N}$ then satisfy the following bounds, where $D_N\ge1$ is used for the second term.
\begin{equation}
\begin{array}{ll}
\frac{2M\sqrt{S_N}\tau_N}{\mathcal A_N}
\le
\frac{4M\sqrt M}{\alpha\sqrt{C_0}}
\quad\mbox{and}\quad
\frac{M\tau_N^2}{\mathcal A_N}
\le
\frac{4M^2}{\alpha C_0}.
\end{array}
\label{eq:tail-normalized-bounds}
\end{equation}
With $C_0=\max\Bigl\{4M,\frac{256M^3}{\alpha^2},\frac{16M^2}{\alpha}\Bigr\}$, each ratio in \eqref{eq:tail-normalized-bounds} is at most $1/4$, and $\mathcal B_N/\mathcal A_N\le1/2$.
Because $\mathcal B_N\le\mathcal A_N/2$, the correlation functional $J_N$ in \eqref{eq:correlation-functional} satisfies $\sqrt{J_N}\ge\mathcal A_N-\mathcal B_N\ge\mathcal A_N/2$ by \eqref{eq:tail-subtraction}.
These constants are used only to establish a uniform threshold and are not optimized.

\section{An inverse-logarithmic bound for a scalar recursion}
\label{app:log}

The following scalar comparison lemma turns the recursion in the proof of Theorem~\ref{thm:abstract-main} into the inverse-logarithmic expected-error bound.
In that application, $m_N=\E[\cE_N]$, where $\cE_N$ is the parameter-averaged error energy in \eqref{eq:error-energy-definitions}.

\begin{lemma}
\label{lem:scalar-log-bound}
Let $\{m_N\}$ be a bounded nonnegative sequence and write $(x)_+:=\max\{x,0\}$.
Suppose that there are constants $b,c>0$ such that, for all sufficiently large $N$,
\begin{equation}
\begin{array}{ll}
m_{N+1}\le m_N-\frac{c}{N+1}
(m_N-\frac{b}{N+1})_+^2.
\end{array}
\label{eq:scalar-log-recursion}
\end{equation}
Then there is a constant $C>0$ such that
\[
\begin{array}{ll}
m_N\le \frac{C}{\log(N+2)}
\end{array}
\]
for all sufficiently large $N$.
\end{lemma}

\begin{proof}
Since $\{m_N\}$ is bounded and nonnegative, choose a positive constant $M_0$ such that $0\le m_N\le M_0$ for every $N$.
An inverse-logarithmic comparison sequence will sharpen the bound $m_N\le M_0$ if the comparison is preserved by \eqref{eq:scalar-log-recursion}.
To carry out this comparison, choose $N_0$ sufficiently large that \eqref{eq:scalar-log-recursion} holds for every $N\ge N_0$ and
\[
\begin{array}{ll}
\frac{2cM_0}{N_0+1}\le1,\quad
2b\le M_0(N_0+1),\quad\mbox{and}\quad 
M_0\bigl(1+\log(N_0+2)\bigr)\ge\frac4c.
\end{array}
\]
With $A:=M_0\bigl(1+\log(N_0+2)\bigr)$, the sequence $y_N:=A/(1+\log(N+2))$ for $N\ge N_0$ starts at $y_{N_0}=M_0\ge m_{N_0}$ and satisfies $0<y_N\le M_0$.
With $\Phi_N(x):=x-\frac{c}{N+1}\bigl(x-\frac{b}{N+1}\bigr)_+^2$, \eqref{eq:scalar-log-recursion} becomes $m_{N+1}\le\Phi_N(m_N)$.
To propagate the comparison $m_N\le y_N$, we will show that $\Phi_N$ is nondecreasing on $[0,M_0]$ and that $\Phi_N(y_N)\le y_{N+1}$.
The definition of $\Phi_N$ gives $\Phi_N(x)=x$ for $0\le x\le b/(N+1)$, while for $b/(N+1)<x\le M_0$, the inequalities $N\ge N_0$ and $2cM_0/(N_0+1)\le1$ give
\[
\begin{array}{ll}
\Phi_N'(x)
=
1-\frac{2c}{N+1}
\bigl(x-\frac{b}{N+1}\bigr)
\ge
1-\frac{2cM_0}{N+1}
\ge0.
\end{array}
\]
Thus $\Phi_N$ is nondecreasing on $[0,M_0]$ for every $N\ge N_0$.

The ratio $(1+\log(N+2))/(N+1)$ decreases with $N$, so the threshold $b/(N+1)$ is at most $y_N/2$ for every $N\ge N_0$, as quantified by
\[
\begin{array}{ll}
\frac{b}{(N+1)y_N}
=
\frac{b(1+\log(N+2))}{A(N+1)}
\le
\frac{b}{M_0(N_0+1)}
\le\frac12.
\end{array}
\]
Since $b/(N+1)\le y_N/2$ and $y_N>0$ for every $N\ge N_0$, the positive part satisfies $(y_N-b/(N+1))_+=y_N-b/(N+1)\ge y_N/2$.
The decrease from $y_N$ to $\Phi_N(y_N)$ therefore satisfies
\[
\begin{array}{ll}
\frac{c}{N+1}
\bigl(y_N-\frac{b}{N+1}\bigr)^2
\ge
\frac{c}{4(N+1)}y_N^2.
\end{array}
\]
The decrease from $y_N$ to $y_{N+1}$, in turn, has the upper bound
\[
\begin{array}{ll}
y_N-y_{N+1}
&=A\bigl(\frac{1}{1+\log(N+2)}
-\frac{1}{1+\log(N+3)}\bigr)\\[2mm]
&=\frac{A\log\bigl(\frac{N+3}{N+2}\bigr)}
{(1+\log(N+2))(1+\log(N+3))}
\le\frac{A}{(N+2)(1+\log(N+2))^2}.
\end{array}
\]
Since $A\ge4/c$, it follows that $y_N-y_{N+1}\le cy_N^2/(4(N+1))$.
Consequently,
\[
\begin{array}{ll}
\Phi_N(y_N)
=
y_N-\frac{c}{N+1}
\bigl(y_N-\frac{b}{N+1}\bigr)^2
\le
y_{N+1}.
\end{array}
\]

The comparison already holds at $N_0$, where $m_{N_0}\le y_{N_0}$.
If $m_N\le y_N$ for some $N\ge N_0$, both arguments lie in $[0,M_0]$, where $\Phi_N$ is nondecreasing, so $\Phi_N(m_N)\le\Phi_N(y_N)$.
The recursion \eqref{eq:scalar-log-recursion} and the bound $\Phi_N(y_N)\le y_{N+1}$ therefore preserve the comparison at $N+1$ through the chain
\[
\begin{array}{ll}
m_{N+1}
\le
\Phi_N(m_N)
\le
\Phi_N(y_N)
\le
y_{N+1}.
\end{array}
\]
Thus $m_N\le y_N = \frac{A}{1+\log(N+2)}$ for every $N\ge N_0$, which proves the claimed bound with $C=A$.
\end{proof}


\small

\begin{thebibliography}{99}\bibitem{BaLi2026}
Y. Ba and Q. Li,
A variable-separation method based on the frequency domain for time-domain Maxwell's equations with random inputs,
Journal of Computational Physics, 548 (2026), Article 114579.

\bibitem{BerkoozHolmesLumley1993}
G. Berkooz, P. Holmes and J. L. Lumley,
The proper orthogonal decomposition in the analysis of turbulent flows,
Annual Review of Fluid Mechanics, 25 (1993), 539--575.

\bibitem{Benner2015}
P. Benner, S. Gugercin and K. Willcox,
A survey of projection-based model reduction methods for parametric dynamical systems,
SIAM Review, 57 (2015), 483--531.

\bibitem{Binev2011}
P. Binev, A. Cohen, W. Dahmen, R. DeVore, G. Petrova and P. Wojtaszczyk,
Convergence rates for greedy algorithms in reduced basis methods,
SIAM Journal on Mathematical Analysis, 43 (2011), 1457--1472.

\bibitem{Buffa2012}
A. Buffa, Y. Maday, A. T. Patera, C. Prud'homme and G. Turinici,
A priori convergence of the greedy algorithm for the parametrized reduced basis method,
ESAIM: Mathematical Modelling and Numerical Analysis, 46 (2012), 595--603.

\bibitem{Chamoin2017}
L. Chamoin, F. Pled, P.-E. Allier and P. Ladev{\`e}ze,
A posteriori error estimation and adaptive strategy for PGD model reduction applied to parametrized linear parabolic problems,
Computer Methods in Applied Mechanics and Engineering, 327 (2017), 118--146.

\bibitem{Chen2024}
L. Chen, Y. Chen, Q. Li and Z. Zhang,
Stochastic domain decomposition based on variable-separation method,
Computer Methods in Applied Mechanics and Engineering, 418 (2024), Article 116538.

\bibitem{ChenZhou2025}
L. Chen, Y. Chen, Q. Li and T. Zhou,
A dynamical variable-separation method for parameter-dependent dynamical systems,
SIAM Journal on Scientific Computing, 47 (2025), A1783--A1808.

\bibitem{CohenDeVore2015}
A. Cohen and R. DeVore,
Approximation of high-dimensional parametric PDEs,
Acta Numerica, 24 (2015), 1--159.

\bibitem{DeVore2013}
R. DeVore, G. Petrova and P. Wojtaszczyk,
Greedy algorithms for reduced bases in Banach spaces,
Constructive Approximation, 37 (2013), 455--466.

\bibitem{FalcoNouy2011}
A. Falc{\'o} and A. Nouy,
A proper generalized decomposition for the solution of elliptic problems in abstract form by using a functional Eckart--Young approach,
Journal of Mathematical Analysis and Applications, 376 (2011), 469--480.

\bibitem{FalcoNouy2012}
A. Falc{\'o} and A. Nouy,
Proper generalized decomposition for nonlinear convex problems in tensor Banach spaces,
Numerische Mathematik, 121 (2012), 503--530.

\bibitem{Ghanem2003}
R. G. Ghanem and P. D. Spanos,
Stochastic Finite Elements: A Spectral Approach,
Dover Publications, Mineola, 2003.

\bibitem{GreplPatera2005}
M. A. Grepl and A. T. Patera,
A posteriori error bounds for reduced-basis approximations of parametrized parabolic partial differential equations,
ESAIM: Mathematical Modelling and Numerical Analysis, 39 (2005), 157--181.

\bibitem{Hesthaven2016}
J. S. Hesthaven, G. Rozza and B. Stamm,
Certified Reduced Basis Methods for Parametrized Partial Differential Equations,
SpringerBriefs in Mathematics, Springer, Cham, 2016.

\bibitem{JiangLi2018}
L. Jiang and Q. Li,
Model reduction method using variable-separation for stochastic saddle point problems,
Journal of Computational Physics, 354 (2018), 43--66.

\bibitem{KammannTroltzschVolkwein2013}
E. Kammann, F. Tr{\"o}ltzsch and S. Volkwein,
A posteriori error estimation for semilinear parabolic optimal control problems with application to model reduction by POD,
ESAIM: Mathematical Modelling and Numerical Analysis, 47 (2013), 555--581.

\bibitem{KunischVolkwein2001}
K. Kunisch and S. Volkwein,
Galerkin proper orthogonal decomposition methods for parabolic problems,
Numerische Mathematik, 90 (2001), 117--148.

\bibitem{KunischVolkwein2002}
K. Kunisch and S. Volkwein,
Galerkin proper orthogonal decomposition methods for a general equation in fluid dynamics,
SIAM Journal on Numerical Analysis, 40 (2002), 492--515.

\bibitem{LiJiang2017}
Q. Li and L. Jiang,
A novel variable-separation method based on sparse and low rank representation for stochastic partial differential equations,
SIAM Journal on Scientific Computing, 39 (2017), A2879--A2910.

\bibitem{LiZhang2020}
Q. Li and P. Zhang,
A variable-separation method for nonlinear partial differential equations with random inputs,
SIAM Journal on Scientific Computing, 42 (2020), A723--A750.

\bibitem{MadayPateraTurinici2002}
Y. Maday, A. T. Patera and G. Turinici,
A priori convergence theory for reduced-basis approximations of single-parameter elliptic partial differential equations,
Journal of Scientific Computing, 17 (2002), 437--446.

\bibitem{Nouy2007}
A. Nouy,
A generalized spectral decomposition technique to solve a class of linear stochastic partial differential equations,
Computer Methods in Applied Mechanics and Engineering, 196 (2007), 4521--4537.

\bibitem{Nouy2010}
A. Nouy,
A priori model reduction through proper generalized decomposition for solving time-dependent partial differential equations,
Computer Methods in Applied Mechanics and Engineering, 199 (2010), 1603--1626.

\bibitem{Peterson1989}
J. S. Peterson,
The reduced basis method for incompressible viscous flow calculations,
SIAM Journal on Scientific and Statistical Computing, 10 (1989), 777--786.

\bibitem{Prudhomme2002}
C. Prud'homme, D. V. Rovas, K. Veroy, L. Machiels, Y. Maday, A. T. Patera and G. Turinici,
Reliable real-time solution of parametrized partial differential equations: Reduced-basis output bound methods,
Journal of Fluids Engineering, 124 (2002), 70--80.

\bibitem{Quarteroni2016}
A. Quarteroni, A. Manzoni and F. Negri,
Reduced Basis Methods for Partial Differential Equations: An Introduction,
Springer, Cham, 2016.

\bibitem{Rozza2008}
G. Rozza, D. B. P. Huynh and A. T. Patera,
Reduced basis approximation and a posteriori error estimation for affinely parametrized elliptic coercive partial differential equations,
Archives of Computational Methods in Engineering, 15 (2008), 229--275.

\bibitem{Sirovich1987}
L. Sirovich,
Turbulence and the dynamics of coherent structures. Part I: Coherent structures,
Quarterly of Applied Mathematics, 45 (1987), 561--571.

\bibitem{VeroyPatera2005}
K. Veroy and A. T. Patera,
Certified real-time solution of the parametrized steady incompressible Navier--Stokes equations: Rigorous reduced-basis a posteriori error bounds,
International Journal for Numerical Methods in Fluids, 47 (2005), 773--788.

\end{thebibliography}
\end{document}